\documentclass[twoside, reqno]{amsart}

\usepackage{enumerate,mathpazo}
\usepackage{euscript} 
\usepackage{mathrsfs}
\usepackage{layout}
\usepackage[pdftex]{graphicx} 
\usepackage{mathrsfs}
\usepackage{amssymb}
\usepackage[dvips]{geometry}
\usepackage{color}
\usepackage[all]{xy}
\usepackage[utf8]{inputenc}

\newtheorem{proposition}{Proposition}[section]
\newtheorem{theorem}[proposition]{Theorem}
\newtheorem{cor}[proposition]{Corollary}
\newtheorem{lemma}[proposition]{Lemma}

\theoremstyle{definition}

\newtheorem{defi}[proposition]{Definition}
\newtheorem{ej}[proposition]{Example} 
\newtheorem{obs}[proposition]{Remark}

\newtheorem{notation}[proposition]{Notation}

\newcommand{\nat}{\mathbb N} 

\newcommand{\re}{\mathbb R} 
\newcommand{\comp}{{\mathbb C}}

\newcommand{\Cob}{\mathrm{Cob}}

\newcommand{\tsf}[1]{\textsf{#1}}

\newcommand{\scr}{\mathscr}

\newcommand{\opnm}[1]{\operatorname{#1}}

\newcommand{\mr}[1]{\mathscr{#1}}

\newcommand{\tm}{\widetilde{M}}

\newcommand{\pr}{\operatorname{pr}}

\usepackage[inline]{asymptote} 

\begin{asydef}
/**
 * @mainpage String Module in Asymptote
 * @author Jorge Andrés Devoto <jdevoto@dm.uba.ar>
 * @date   Sun Sep 18 12:10:30 2011
 * 
 * @brief  Module for Drawings of String theory
 * @version 0.2
 *
 * @section intro_sec Introduction
 *
 * Contiene macros para dibujar en asymptote
 * los dibujos tipicos de TFT en 2d
 * abierta y cerrada. Pueden ser de cualquier tamaño
 * y posicion:
 *
 *   Closed Strings
 *
 *   - Incoming and Outgoing circles
 *   - Incoming and Outgoing semicircles
 *   - Cilinders
 *   - Aniquilacion y creacion de cuerdas cerradas
 *   - Pair of Pants
 *
 *   Open Strings
 *
 *   - Cuerda abiertas con bordes con labels
 *   - Creación y aniquilación de cuerdas abiertas.
 *   - Pair of pants de cuerdas abiertas
 *   - Transición open-closed
 *
 * @section inst_sec Instalacion
 *
 * Para instalar en linux poner en un directorio .asy en el home
 * el archivo strings.asy.
 *
 * Poner la documentación en donde se guarde normalmente la documentación.
 *
 * @section uso_sec Uso
 *
 * Poner en un archivo LaTeX o asy la instrucción
 *
 * import strings;
 * 
 * @section changes_sec Cambios
 */

/**
 * @file   strings_02.asy
 * @author Jorge Andrés Devoto <jdevoto@dm.uba.ar>
 * @date   Sun Sep 18 12:10:30 2011
 * 
 * @brief  Module for Drawings of String theory
 * @version 0.2
 *
 * Contiene macros para dibujar de cualquier tamaño
 * y posicion:
 *
 *   - Incoming and Outgoing circles
 *   - Incoming and Outgoing semicircles
 *   - Cilinders
 *   - Creation operators
 *   - Pair of Pants
 */

unitsize(2 cm);

struct Circle
{
  /** 
   * 
   * Draw boundary circles.
   * 
   * Usage: bcircle b = bcircle(center, upoint);
   *
   * @param center: The center of the ellipse
   * @param upoint: The semiaxis of the ellipse
   * 
   * @return c: the object
   */
  
  //! center of the ellipse
  pair center;
  //! Punto del eje semimajor segun wikipedia
  pair upoint;
  //! The difference upoint - center
  pair maxis;
  //! Ortogonal to maxis con orientacion positiva
  pair saxis;
  //! La figura
  path fig;
  //! El angulo que hay que rotar
  real angulo;

  static Circle Circle(pair center, pair upoint)
  {
    Circle c = new Circle;

    // Construye un circulo estandar
    
    pair z0 = (0, 1);
    pair z1 = (0.5, 0);
    pair z2 = (0, -1);
    pair z3 = (-0.5, 0);

    path p1 = z3{up} .. z0 .. {down}z1;
    path p2 = z1{down} .. z2 .. {up}z3;
    path stdcircle = (p2 & p1) .. cycle;

    // Calcula parametros de translacion y rotacion

    c.center = center;
    c.upoint = upoint;
    c.maxis = upoint  - center;
    c.saxis = (c.maxis.y, -1 * c.maxis.x);
    
    real dist = length(c.maxis);  // length of major axis
    real ang;                   // angle of rotation

    if((angle(c.saxis)) < 0) {
	ang = angle(c.saxis) + 2 * pi;
      }
    else{
      ang = angle(c.saxis);
    }
    
    ang = degrees(ang);
    c.angulo = ang;
    
    // Hace las transformaciones

    c.fig = (shift(center)*(rotate(ang)*(scale(dist)*stdcircle)));
    return c;
  }
}

from Circle unravel Circle;

struct LScircle{
  /** 
   * Draw a Left semicircle (
   * 
   * @param center = the center of the semcircle
   * @param upoint = the semiaxis of the semicircle
   * 
   * 
   * @return a semicircle object
   */
  
  pair center;
  pair upoint;
  path fig;

  static LScircle LScircle(pair center, pair upoint){
    LScircle c = new LScircle;

    // Standard Left Semicircle

    pair z0 = (0, -1);
    pair z1 = (-0.5, 0);
    pair z2 = (0, 1);

    path std = ((z0{left} .. z1 .. {right}z2));

    // Transforms

    // Calcula parametros de translacion y rotacion

    pair maxis = upoint  - center;
    pair saxis = (maxis.y, -1 * maxis.x);
    real dist = length(maxis);  // length of major axis
    real ang;                   // angle of rotation
    
    if(abs(angle(saxis)) > pi/2){
      ang  = angle(saxis) - pi/2;
    }
    else{
      ang = angle(saxis);
    }

    ang = degrees(ang);

    // Hace las transformaciones

    c.fig = (shift(center)*(rotate(ang)*(scale(dist)*std)));

    return c;
  }
}

from LScircle unravel LScircle;

struct RScircle{
  /** 
   * Draw a Right semicircle )
   * 
   * @param center = the center of the semcircle
   * @param upoint = the semiaxis of the semicircle
   * 
   * 
   * @return a semicircle object
   */
  
  pair center;
  pair upoint;
  path fig;

  static RScircle RScircle(pair center, pair upoint){
    RScircle c = new RScircle;

    // Standard Left Semicircle

    pair z0 = (0, -1);
    pair z1 = (0.5, 0);
    pair z2 = (0, 1);

    path std = z0{right} .. z1 .. {left}z2;

    // Transforms

    // Calcula parametros de translacion y rotacion

    pair maxis = upoint  - center;
    pair saxis = (maxis.y, -1 * maxis.x);
    real dist = length(maxis);  // length of major axis
    real ang;                   // angle of rotation
    
    if(abs(angle(saxis)) >  pi/2){
      ang  = angle(saxis) - pi/2;
    }
    else{
      ang = angle(saxis);
    }

    ang = degrees(ang);

    // Hace las transformaciones

    c.fig = (shift(center)*(rotate(ang)*(scale(dist)*std)));

    return c;
  }
}

from RScircle unravel RScircle;

// Functions that return paths and pictures

path bcircle(pair center, pair upoint){
  /** 
   * bcircle draws a circle.
   *
   * Use: path d = bcircle((110, 0),(2,3));
   * 
   * @param center The center of the circle
   * @param upoint The upper point in the main axis
   * 
   * @return a path object
   */
  Circle c = Circle(center, upoint);
  return c.fig;
}

path lscircle(pair center, pair upoint){
  /** 
   *
   * lscircle draws a left semicircle (
   *
   * Usage: path d = lscircle((110, 0),(2,3));
   * 
   * @param center The center of the circle
   * @param upoint The upper point in the main axis
   * 
   * @return a path object
   */
  LScircle c = LScircle(center, upoint);
  return c.fig;
}

path rscircle(pair center, pair upoint){
  /** 
   *
   * rscircle draws a right semicircle )
   *
   * Usage: path d = rscircle((110, 0),(2,3));
   * 
   * @param center The center of the circle
   * @param upoint The upper point in the main axis
   * 
   * @return a path object
   */
  RScircle c = RScircle(center, upoint);
  return c.fig;
}

// Higher structures

struct Scilinder 
{
  /** 
   * Draw a cilinder
   * 
   * @param icenter The centre of the incoming circle
   * @param ocenter The centre of the incoming circle
   * @param radius The radius of the cilinder
   * 
   * @return A structure
   */
  
  // The data for the incoming circle
  pair icenter;
  pair iupoint;
  // The data for the outgoing circle
  pair ocenter;
  // The picture object
  picture pic;
  // General colours
  pen colour2 = mediumgray;
  pen colour3 = white;

  static Scilinder Scilinder(pair icenter, pair ocenter, real radius) 
  {

    Scilinder c = new Scilinder;

    // Pens
    pen mg = lightgray;
    pen wh = white;
    pen bc = black; 
    
    // Missing data

    pair main_axis = (ocenter - icenter);
    real length_main_axis = length(main_axis);
    pair sec_axis = radius * (unit(I * main_axis));
    
    // form the iupoint

    pair iupoint = icenter + sec_axis;
    pair oupoint = ocenter + sec_axis;
     
    path ic = bcircle(icenter, iupoint);
    path odsc = lscircle(ocenter, oupoint);
    path isc = lscircle(icenter, iupoint);
    path ub = iupoint -- oupoint;
    path db = ((2 * icenter) - iupoint) -- ((2 * ocenter) - oupoint);
    
    path osc = reverse(rscircle(ocenter, oupoint));

    path border = (isc & ub & osc & reverse(db))..cycle;

    filldraw(c.pic, border, mg);
    filldraw(c.pic, ic, wh);
    draw(c.pic, odsc, dashed);
        
    return c;
  }
}

from Scilinder unravel Scilinder;

void cilinder(pair icenter, pair ocenter,
	      real radius, picture p = currentpicture)
{
  /** 
   * Draw a cilinder
   *
   * Usage:  cilinder(icenter, ocenter, radius, picture (opcional))
   *
   * @param icenter The incoming center
   * @param ocenter The outgoing center
   * @param radius  The radius
   * @param p A picture object. Default current picture
   *
   * @return void
   */
  Scilinder c = Scilinder(icenter, ocenter, radius);
  add(p,c.pic);
}

struct Creation{
  /** 
   * Draw a creation-destruction picture
   * 
   * @param vacuum: The point of creation
   * @param ocenter: Center of the out/ingoing circle 
   * @param oupoint: Semiaxis of the  out/ingoing circle 
   * 
   * @return The structure
   */
  
  // The picture

  picture pic;
  
  // The position of the vacuum

  pair vacuum;

  // Parameters for the out circle

  pair ocenter;
  pair upoint;

  static Creation Creation(pair vacuum, pair ocenter,
			   pair oupoint)
  {
    Creation c = new Creation;

    // pens
    
    pen wh = white;
    pen mg = lightgray;

    // Determine axis;
    pair main_axis = (ocenter - vacuum);
    pair sec_axis = I * main_axis;

    // Downpoint;

    pair dpoint = (2 * ocenter) - oupoint;

    path outc = bcircle(ocenter, oupoint);
    path rout = rscircle(ocenter, oupoint);

    // Make the border

    path bb = oupoint{(-1) * main_axis} .. vacuum .. dpoint{main_axis};

    path border = (rout & bb) .. cycle;
    
    filldraw(c.pic, border, mg);
    filldraw(c.pic, outc, wh);

    return c;
  }
}

from Creation unravel Creation;

void creation(pair vacuum, pair ocenter,
	      pair oupoint, picture pic = currentpicture)
{
  /** 
   * Draw the creation diagram
   *
   * Usage: creation(vacuum, ocenter, oupoint);
   *
   * @param vacuum: The point of creation
   * @param ocenter: Center of the out/ingoing circle  
   * @param oupoint: Semiaxis of the  out/ingoing circle  
   * 
   * @return void
   */
  
  Creation c = Creation(vacuum, ocenter, oupoint);
  add(pic, c.pic);
}

// Pair of pants

struct PairOfPants
{
  /** 
   * @class Estructura que da un par de pantalones
   *
   * Hay tres círculos numerados circulo 1, 2, 3
   * El círculo 1 es el de la cintura. Los parámetros de
   * entrada determinan un eje que corresponde al eje vertical del
   * pantalon. El círculo 2 es el que queda del mismo lado del eje que
   * el punto iuppoint.
   * 
   */

  //! The center of the one incoming hole
  pair icenter;

  //! The semiaxis of the incoming hole
  pair iuppoint;

  //! The radius is the overall scale
  real radius;
  
  //! The main axis. Vertical axis of the pants.
  pair main_axis;

  //
  // The other two circles
  //
  
  // circle 2

  //! Center of the second circle
  pair icenter2;
  //! The semiaxis of the second circle
  pair iuppoint2;

  // circle 3

  //! Center of the third circle
  pair icenter3;
  //! The semiaxis of the third circle
  pair iuppoint3;

  // General colours
  pen colour2 = lightgray;
  pen colour3 = white;

  //! The picture object
  picture pic;

  path border;

  /** 
   * Constructor que define un par de pantalones
   * 
   * @param icenter El centro de la "cintura"
   * @param iupoint El semieje de la "cintura
   * @param b Dice si el circulo de la cintura es blanco (si b = true)
   *        o gris (si b = false)
   * 
   * @return un objeto tipo PairOfPants
   */
  static PairOfPants PairOfPants(pair icenter, pair iupoint, bool b)
  {
    
    PairOfPants c = new PairOfPants;

    // Axis of the ingoing circle
    c.main_axis = I * (icenter - iupoint);

    // Radius
    c.radius = length((iupoint - icenter));

    //
    // Draw a standard border of a pair of pants
    //

    path stdIncoming = lscircle((0, 0),(0,1));
    path stdUpOutGoing = reverse(rscircle((5, 2), (5, 3)));
    path stdDownOutGoing = reverse(rscircle((5, -2), (5, -1)));

    path Union = reverse((5, -1){left} .. (4, 0) .. (5, 1){right});
    path Upborder = reverse((5, 3){left} .. (4, 2.9) .. (1.75, 1.2) .. (1, 1.1) .. (0, 1){left});
    path Dborder = reflect((-2, 0), (3, 0)) * reverse(Upborder);

    // Paths que completan en dibujo

    path stdInBorder = rscircle((0,0),(0,1));
    path b1 = lscircle((5,2),(5,3));
    path d1 = lscircle((5,-2),(5,-1));
    
    path b2 = bcircle((5, 2), (5,3));
    path d2 = bcircle((5, -2), (5,-1));

    // Dibujo del borde
    
    path std_border = (((stdIncoming & Upborder) & stdUpOutGoing) & Union);
    std_border = ((std_border & stdDownOutGoing) & Dborder) .. cycle;

    // Compute the rotation angle

    real angulo = degrees(c.main_axis);

    // Draw the border

    c.border = (shift(icenter)*(rotate(angulo)*(scale(c.radius)*std_border)));

    // Draw the border of the figure

    filldraw(c.pic, c.border, c.colour2);

    // Determina los parametros de los otros agujeros

    c.icenter2 = c.radius * (5, 2) + icenter;
    c.iuppoint2 = c.radius * (5, 3) + icenter;
    
    c.icenter3 = c.radius * (5, -2) + icenter;
    c.iuppoint3 = c.radius * (5, -1) + icenter;
    
    // Mirar que agujeros hay que llenar

    if (b == true){
      path a = bcircle(icenter, iupoint);
      filldraw(c.pic, a, c.colour3);

      path b = (shift(icenter)*(rotate(angulo)*(scale(c.radius)* b1)));
      draw(c.pic, b, dashed);
      
      path d = (shift(icenter)*(rotate(angulo)*(scale(c.radius)* d1)));
      draw(c.pic, d, dashed);
    }
    else
      {
	path a = (shift(icenter)*(rotate(angulo)*(scale(c.radius)*stdInBorder)));
	draw(c.pic, a, dashed);

	path a2 = (shift(icenter)*(rotate(angulo)*(scale(c.radius)* b2)));
	filldraw(c.pic, a2, c.colour3);

	path a3 = (shift(icenter)*(rotate(angulo)*(scale(c.radius)* d2)));
	filldraw(c.pic, a3, c.colour3);
      }

    return c;
  }
}

from PairOfPants unravel PairOfPants;

/** 
 * Función que implementa la actual creacion de un par de pantalones.
 *
 * Usage:
 *
 * pair_of_pants(icenter, iupoint, b, pic)
 *
 * @param icenter El centro de la "cintura"
 * @param iupoint El semieje de la "cintura
 * @param b Dice si el circulo de la cintura es blanco (si b = true)
 *        o gris (si b = false)
 * @param pic El "picture" donde dibujar el par. El valor default es
 * currentpicture. 
 *
 */
void pair_of_pants(pair icenter, pair iupoint, bool b, picture pic = currentpicture)
{ 
  PairOfPants c = PairOfPants(icenter, iupoint, b);
  add(pic, c.pic);
}

/** 
  * OpenString:
  *
  * Structure for open strings
  * 
  * @param a The  labelled a point
  * @param b The labelled b point
  * 
  * @return The structure
  */
struct OpenString{

  //! The labelled "a" end of the string
  pair in_a;
  //! The labelled "b" end of the string
  pair out_b;
  //! The path
  path pp;
  //! The figure into a picture environment
  picture pic;

  static OpenString OpenString(pair a, pair b){

    OpenString c = new OpenString;

    // Draw a standard open string
    path std_open_string = (0,0){right} .. (0.333, 0.15) .. (0.666, -0.15) .. {right}(1, 0);

    // compute the resize parameters

    real len = length(b - a);
    real angulo = degrees(b - a);
    
    // resize string

    path str = shift(a) * (rotate(angulo) * (scale(len) * std_open_string));

    c.pp = str;
    
    // Draw the string

    dot(c.pic, b, red);
    dot(c.pic, a, blue);
    draw(c.pic, str);
    return c;
    }
}

from OpenString unravel OpenString;

/** 
 * openstring
 *
 * Función para dibujar una cuerda abierta con labelled ends
 *
 * Usage: openstring(pic, a, b)
 * 
 * @param pic (Opcional) The picture where to draw the string. Default currentpicture.
 * @param a A pair. The labelled a end of the string
 * @param b A pair. The labelled b end of the string
 */
void openstring(picture pic = currentpicture, pair a, pair b){

  OpenString c = OpenString(a, b);
  add(pic, c.pic);
}

/** 
 * OpenCreation
 *
 * Estructura que define un diagrama de creación
 * de una cuerda abierta
 * 
 * @param a Define el punto con label a de la cuerda
 * @param b Define el punto con label a de la cuerda
 * 
 * @return 
 */
struct OpenCreation{
  
  //! The labelled "a" end of the created string
  pair in_a;
  //! The labelled "b" end of the created string
  pair out_b;
  //! The figure into a picture environment
  picture pic;

  static OpenCreation OpenCreation(pair a, pair b){

    OpenCreation c = new OpenCreation;

    // Draw a standard open creation figure

    path std_open_string = (0,0){right} .. (0.333, 0.06) .. (0.666, -0.06) .. {right}(1, 0);

    // compute the position of the vacuum

    pair vacuum = (0.5, 0.8);

    // define the open creation figure

    path border_std_open_creation = (1,0){up} .. vacuum .. {down}(0,0);

    path std_open_creation = (std_open_string .. border_std_open_creation) .. cycle;

    // compute the resize parameters

    real len = length(b - a);
    real angulo = degrees(b - a);
    
    // resize string

    path border = shift(a) * (rotate(angulo) * (scale(len) * border_std_open_creation)); 
    path str = shift(a) * (rotate(angulo) * (scale(len) * std_open_creation));

    // Draw the string

    filldraw(c.pic, str, lightgray);
    dot(c.pic, b, red);
    dot(c.pic, a, red);
    draw(c.pic, border, red);

    // return

    return c;
  }
}

from OpenCreation unravel OpenCreation;

/** 
 * Función que dibuja un diagrama de creación de una cuerda abierta
 * de extremos con labels a y b
 * 
 * @param pic Un picture donde se dibuja el diagrama. El valor default
 * es currentpicture. Casi no hace falta usarlo.
 * @param a La posición del extremo a.
 * @param b La posición del extremo b.
 */
void opencreation(picture pic = currentpicture, pair a, pair b){

  OpenCreation c = OpenCreation(a, b);
  add(pic, c.pic);
}

struct OpenPairOfPants{

  //! The center of the one incoming hole
  pair icenter;

  //! The semiaxis of the incoming hole
  pair iuppoint;

  //! The radius is the overall scale
  real radius;
  
  //! The main axis. Vertical axis of the pants.
  pair main_axis;

  // General colours
  pen colour2 = lightgray;
  pen colour3 = white;

  //! The picture object
  picture pic;

  path border;

  static OpenPairOfPants OpenPairOfPants(pair icenter, pair iupoint)
  {

    OpenPairOfPants c = new OpenPairOfPants;

    // Axis of the ingoing circle
    c.main_axis = I * (icenter - iupoint);

    // Radius
    c.radius = length((iupoint - icenter));

    // Standard Object

    path incoming = (0, -1) .. (-0.1, -0.33) .. (0.1, 0.33) .. (0,1);

    path upoutgoing = reverse((5, 1) .. (4.9, 1.66) .. (5.1, 2.33) .. (5, 3));
    path downoutgoing = (5, -1) .. (4.9, -1.66) .. (5.1, -2.33) .. (5, -3);

    path Union = reverse((5, -1){left} .. (4, 0) .. (5, 1){right});
    path Upborder = reverse((5, 3){left} .. (4, 2.9) .. (1.75, 1.2) .. (1, 1.1) .. (0, 1){left});
    path Dborder = reflect((-2, 0), (3, 0)) * reverse(Upborder);

    // Dibujo del borde
    
    path std_border = (((incoming & Upborder) & upoutgoing) & Union);
    std_border = ((std_border &  downoutgoing) & Dborder) .. cycle;

    // Compute the rotation angle

    real angulo = degrees(c.main_axis);

    // Draw the border

    c.border = (shift(icenter)*(rotate(angulo)*(scale(c.radius)*std_border)));

    path ub = (shift(icenter)*(rotate(angulo)*(scale(c.radius)* Upborder)));
    path db = (shift(icenter)*(rotate(angulo)*(scale(c.radius)* Dborder)));
    path uu = (shift(icenter)*(rotate(angulo)*(scale(c.radius)* Union)));
    // Draw the border of the figure

    filldraw(c.pic, c.border, c.colour2);
    draw(c.pic, ub, red);
    draw(c.pic,  db, blue);
    draw(c.pic, uu, green);
    
    return c;
  }
}

from OpenPairOfPants unravel OpenPairOfPants;

void openpairofpants(pair a, pair b, picture c = currentpicture)
{
  OpenPairOfPants d = OpenPairOfPants(a, b);
  add(c, d.pic);
}

/** 
 * OpenClosed: Estructura que representa una transición cerrada abierta
 * 
 * @param open:  The center of the open string
 * @param closed:  The center of the closed string
 * @param radius: The radius of the string 
 * 
 * @return 
 */
struct OpenClosed{

  //! The center of the open string
  pair open;

  //! The center of the closed string
  pair closed;

  //! The radius of the string
  real radius;

  //! The whole picture
  picture pic;
      
  static OpenClosed OpenClosed(pair open, pair closed, real radius){

    OpenClosed c = new OpenClosed;

    // Compute the parameters

    real len = length(closed - open);
    real angle = degrees(closed - open);

    // Components of the figure

    OpenString str_in = OpenString((0, -0.5 * radius), (0, 0.5 * radius));
    LScircle lclos = LScircle((len, 0), (len, 0.5 * radius));
    RScircle rclos = RScircle((len, 0), (len, 0.5 * radius));

    // The OutBorder
    
    path out_border = str_in.pp .. ((0, 0.5 * radius) -- (len, 0.5 * radius));
    out_border = out_border .. reverse(rclos.fig);
    out_border = (out_border .. ((len, -0.5 * radius) -- (0, -0.5 * radius))) .. cycle;

    // The in border

    path l = (0, -0.5 * radius){right} .. (0.30 * len, -0.40* radius) .. (0.40* len, -0.3 * radius) .. {up}(0.48 * len, 0);
    l = l .. (reflect((0,0), (len, 0)) * (reverse(l)));
    
    path inborder = (0, -0.5 * radius){right} .. (0.30 * len, -0.40* radius) .. (0.40* len, -0.3 * radius) .. {up}(0.48 * len, 0);
    inborder = inborder .. (reflect((0,0), (len, 0)) * (reverse(inborder)));
    
    inborder = inborder & ((0, 0.5 * radius) -- (len, 0.5 * radius));
    inborder = inborder & reverse(rclos.fig);
    inborder = (inborder & ((len, -0.5 * radius) -- (0, -0.5 * radius))).. cycle;

    // Transform

    out_border = shift(open) * (rotate(angle) * out_border);
    inborder = shift(open) * (rotate(angle) * inborder);
    path ll =  shift(open) * (rotate(angle) * lclos.fig);
    l = shift(open) * (rotate(angle) * l);

    // Draw
    
    filldraw(c.pic, out_border, lightgray);
    filldraw(c.pic, inborder, mediumgray);

    draw(c.pic, ll, dashed);
    draw(c.pic, l, red);
    
    return c;
    
  }
}

from OpenClosed unravel OpenClosed;

/** 
 * Función: openclosed
 *
 * Uso: openclosed(open, closed, radius, picture)
 * 
 * @param open:  The center of the open string
 * @param closed: The center of the closed string 
 * @param radius: The radius of the closed string 
 * @param pic : The picture
 *
 */
void openclosed(pair open, pair closed, real radius, picture pic=currentpicture)
{
  
  OpenClosed d = OpenClosed(open, closed, radius);
  add(pic, d.pic);
}

\end{asydef}
\usepackage[utf8]{inputenc}

\title{2-vector bundles, D-branes and Frobenius Manifolds}
\author{Anibal Amoreo}
\address{Math. Dept. Facultad de Ingenier\'\i a, UBA, Paseo Col\'on
  850, Buenos Aires, Republica
  Argentina}
\email{anibal.amoreo@gmail.com}
\author{Jorge A. Devoto}
\address{Math. Dept. FCEN, UBA, Ciudad
  Universitaria \\ pabellon 1, 1428, Buenos Aires, Republica
  Argentina}
\email{jdevoto@dm.uba.ar}
\date{February 18, 2013}
\keywords{2-vector spaces, Frobenius manifolds}
\subjclass{55R65, 53C05}

\begin{document}

\begin{abstract}
  We show that if $M$ is a Frobenius manifold of dimension $n$ such
  that $T_{x} M$ is semisimple for every $x \in M$, then there exists
  a canonical 2-vector bundle $\mathcal{B}$ over $M$ of rank $n$. This
  2-vector bundle encodes the information about the maximal category
  of $D$-branes associated to the open closed topological field
  theories defined by the Frobenius algebras $T_{x} M$. In particular
  this construction answers a conjecture of Graeme Segal
  in~\cite{segal07:_what_is_ellip_objec}. We also explain the relation
  of the labels of the $D$-branes to Azumaya algebras and twisted
  vector bundles on the spectral cover $S$ of $M$.
\end{abstract}

\maketitle

\section{Introduction}
\label{sec:introduction}

The aim of the present work is to give a positive answer to a remark
of G. Segal, in~\cite{segal07:_what_is_ellip_objec}, about a possible
relation between 2-vector bundles and the moduli space of topological
field theories.
 
The geometric objects involved in this remark -- 2-vector bundles --
are a topological generalisation of the algebraic notion of 2-vector
spaces. The notion of 2-vector spaces was introduced by M. Kapranov
and V. Voedvodski in~\cite{kn:kv}. This notion is a categorification
of the concept of vector space. The idea of 2-vector bundles was
proposed as a geometric model for elliptic cohomology. Constructions
and definitions of 2-vector bundles were proposed by
J.L. Brylinsky~\cite{brylinski:_catvb} and
N. Baas, B.  Dundas and J Rognes~\cite{baas04:_two}.

The ideas behind these constructions are related to physics, in
particular string theory. At the beginning of the 90's it was
suggested, see for
example~\cite{freed94:_higher_algeb_struc_and_quant}, that some form
of 2-vector spaces should be attached to the endpoints of an open
string. In \cite{segal07:_what_is_ellip_objec} Graeme Segal suggested
that there might be a relation between 2-vector bundles and the moduli
space of topological field theories. 

Two dimensional topological field theories can be algebraically
described in terms of \emph{commutative Frobenius algebras}. In the
case of commutative semisimple Frobenius algebras G. Moore and
G. Segal~\cite{moore_segal1, al.]09:_diric} found a geometric
description of these algebras as the algebras of functions on finite
sets equipped with a measure. These finite sets play the role of
spacetimes in the theory. It is then natural to think that a smoothly
varying family of 2d-topological field theories is a pair $(\pi: S \to
M, f)$ formed by a smooth manifold $M$ with a fixed finite sheeted
covering space $S \to M$ and a function $f: S \to \re$. The points $x$
of $M$ parametrise the topological field theories of the family
defined by the fibres $\pi^{-1} (x)$ with the measure induced by
$f$. This type of structure appeared in the work of Saito about
unfoldings of singularities. The structure reappeared in the notion of
a \emph{Frobenius manifold} defined by
B. Dubrovin~\cite{dubrovin:_2dtft}.  A Frobenius manifold is basically
a manifold $M$ with the property that the tangent spaces $T_{x} M$
have a structure of a Frobenius algebra $\forall x \in M$.  Frobenius
manifolds define a geometric model for the solutions of WDVV
equations. These equations capture the deformations of topological
conformal field theories. When $T_{x} M$ is semisimple for every $x
\in M$, then $M$ has a canonically associated covering space $S \to M$
called the \emph{spectral cover}. This covering space has a natural
function on it which provides the measure. This fact provides the connection
between the viewpoints of Moore and Segal and the definition given by
Dubrovin.

The plan of this paper is the following.  In Section
\ref{sec:frob-manif-d} we recall some basic facts and definitions
about open and closed 2d-topological field theories, Frobenius
algebras, Calabi-Yau categories and Frobenius manifolds, we also
define Cardy categories. In the next section we describe 2-vector
spaces, 2-vector bundles and twisted vector bundles. We shall give a
summary of Moore and Segal description of the maximal category of
$D$-branes in an open and closed topological field theory. In section
\ref{sec:calabi-yau-fibr} we introduce the notion of Cardy fibrations
over Frobenius manifolds.. Then in Section \ref{sec:algebr-prop-maxim}
we give a local characterisation of maximal Cardy fibrations and we
show that maximal Cardy fibrations define two vector bundles. In the
next section, Section \ref{sec:maxim-categ-bran}, we show that some
sectors of a maximal Cardy fibration are related to twisted vector
bundles and Azumaya algebras over the spectral cover of a Frobenius
manifold.

\medskip

\section{Topological D-branes and Frobenius manifolds}
\label{sec:frob-manif-d}

\subsection{Open and closed $2d$ topological field theories}
\label{sec:topol-field-theor}

Let us roughly describe the basic aspects of $2d$ topological field
theories (TFT). Precise definitions for closed theories can be found,
for example, in the article of
L. Abrams~\cite{abrams96:_two_dimen_topol_quant_field}; the references
for the relevant definitions for open and closed topological
field theories are the article of G. Moore and
G. Segal~\cite{moore_segal1} and chapter 2 of \cite{al.]09:_diric}.

The description of a closed TFT is as follows: Let $\Cob$
denote the (pseudo) category whose objects are closed, oriented,
1-dimensional manifolds. We shall consider the empty set as an object
of $\Cob$. If $N$ and $M$ are objects of $\Cob$, then
a morphism $N \to M$ is an ordered pair $(\Sigma, \, \phi)$ formed by
a compact two dimensional oriented manifold $\Sigma $ and an
orientation preserving diffemorphism $\phi: \partial \Sigma \to N
\sqcup -M$, where $-M$ denotes the manifold $M$ with the opposite
orientation. A convenient way to describe a cobordism is to use the
following diagram
$
  N \rightarrow \Sigma \leftarrow -M.
$
Two cobordisms $(\Sigma, \phi)$ and $(\Sigma_{1}, \psi)$ are to be
identified if there is an orientation preserving diffeomorphism
$\alpha: \Sigma \to \Sigma_{1}$ such that the diagram
\begin{displaymath}
  \xymatrix{
                       & & \Sigma \ar[dd]_{\alpha} & &           \\
  N \ar[rru] \ar[rrd]  & &                        & & -M \ar[llu] \ar[lld] \\
                       & & \Sigma_{1}              & &
                     }
\end{displaymath}
commutes. Any object of $\Cob$ is diffeomorphic to a disjoint union of
copies of the standard circle $S^{1}$ and the empty set. If $C_{1}$
and $C_{2}$ are objects of $\Cob$ and $\Sigma$ is a cobordism between
them, the circles in $C_{1}$ are called \emph{ingoing} and the circles
in $C_{2}$ are called \emph{outgoing}.  The morphisms in $\Cob$ are
generated by the following figures, where the outgoing circles are
written to the right

\begin{figure}[http]
  \centering
\includegraphics[width=5cm]{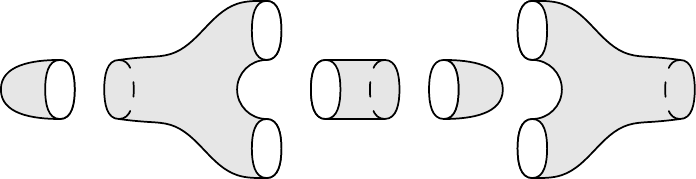}
  \caption{Generators of morphisms in $\Cob$.}
  \label{fig:1}
\end{figure}
The category $\Cob$ has a monoidal structure induced by the disjoint
union of manifolds.

\begin{defi}
  \label{def:1}
  Let $\mathrm{Vect}$ denote the category of vector spaces over $\comp$. A
  \emph{2d-closed topological field theory} (or TFT) is a functor $F:
  \Cob \to \mathrm{Vect}$ such that
  \begin{equation}
    \label{eq:1}
    F(C_{1} \sqcup C_{2}) = F(C_{1}) \otimes F(C_{2}).
  \end{equation}
\end{defi}
A closed topological field theory is determined by two vector spaces:
The vector space $A = F(S^{1})$ and the vector space
$F(\emptyset)$. The multiplicative condition, given by
equation~\eqref{eq:1}, implies that $F(\emptyset)$ is the ground field
$\comp$. On the other hand the generators of the morphisms in
$\Cob$ induce on $A$ the following structure
\begin{figure}[http]
  \centering
\includegraphics[width=5cm]{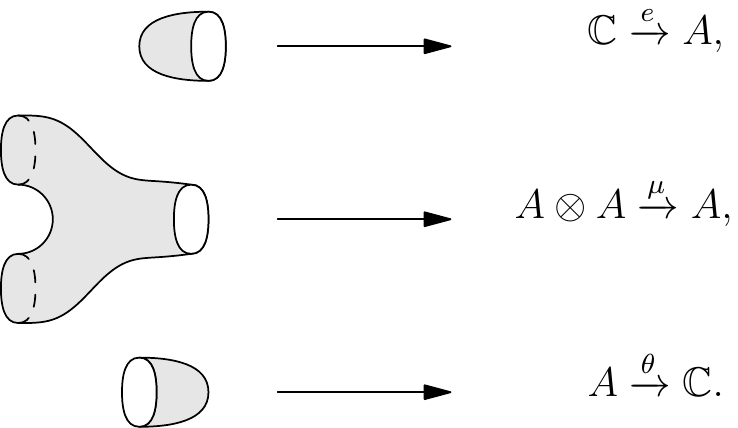}
  \caption{Structure of $A$.}
  \label{fig:2}
\end{figure}

We shall assume that the homomorphism associated to the cylinder is
the identity.  The morphisms in diagram~\ref{fig:2} must satisfy
compatibility conditions.  The algebraic structure induced on $A$ is
the structure of a \emph{Frobenius algebra}-- the precise definition
is in Section~\ref{sec:frobenius-algebras}.

Categories of branes are obtained when one considers a bigger
cobordism category, namely open and closed cobordism. The (pseudo)
category $\mathrm{Ocob}$ is the category where the objects are one
dimensional, compact, oriented manifolds with (possibly empty) boundary. If the
boundary is non-empty we shall suppose that each connected component
of the boundary is labelled by an element of a fixed set
$\mathcal{B}$. The set $\mathcal{B}$ is called the \emph{set of
  boundary conditions}. Any element of $\mathrm{Ocob}$ is
diffeomorphic to a disjoint union of elements of the form given in
figure~\ref{fig:3}.

\begin{figure}[http]
\centering
\includegraphics[width=5cm]{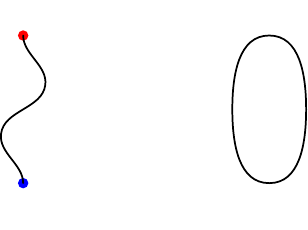}
\caption{Basic elements of $\mathrm{Ocob}$.}
\label{fig:3}
\end{figure}

A cobordism $\Sigma$ between two objects $C_{0}$ and $C_{1}$ of Ocob
is an oriented 
surface whose boundary  consists of three parts $\partial \Sigma =
C_{0} \cup C_{1} \cup C_{cstr}$. The part $C_{cstr}$ is called the
\emph{constrained boundary} and is a cobordism from  $\partial C_{0}$
to $\partial C_{1}$. The components of $C_{cstr}$ are labelled in a
way compatible with the labelling of $\partial C_{0}$ and $\partial
C_{1}$. See figure~\ref{fig:4}.
\begin{figure}[http]
  \centering
\includegraphics[width=3cm]{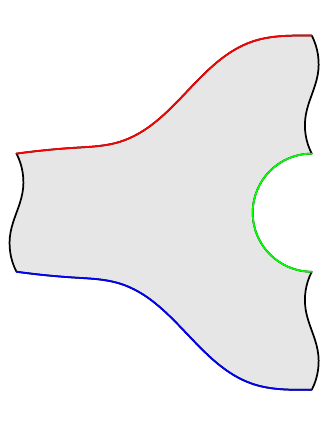}
  \caption{Open cobordism}
  \label{fig:4}
\end{figure}

\begin{defi}
  \label{def:2}
  An \emph{open and closed topological field theory} is a functor $F:
  \mathrm{Ocob}  \to \mathrm{Vect}$ satisfying the multiplicative
  axiom~\eqref{eq:1}. 
\end{defi}
An open and closed topological field theory is algebraically described by
a certain class of ``self-dual'' categories which we shall call 
\emph{Cardy categories}, see Section~\ref{sec:calabi-yau-categ}. We
shall write $E_{ab}$ for the image of the interval with labels $a$ and
$b$. See figure~\ref{fig:3}.

\medskip

% Hasta aqui version final 2014-06-30

\subsection{Frobenius algebras}
\label{sec:frobenius-algebras}

We will recall here some basic facts about Frobenius algebras.  A
general reference about Frobenius algebras over a field
is~\cite{abrams96:_two_dimen_topol_quant_field}. We shall need also
some basic definitions about Frobenius algebras over rings --
see for example~\cite{nakayama4}. Let $R$ denote a
commutative ring. If $A$ is an $R$-module we shall write $A^{*} =
\mathrm{Hom}_{R} (A, R)$ for the dual module. If $A$ is a $R$-algebra,
then $A^{*}$ has a natural structure of a left $A$-module given by
\begin{equation}
  \label{eq:2}
  (a \varphi) (b) = \varphi (ab),
\end{equation}
for $a, b $ in $A$ and $\varphi$ in $A^{*}$. We shall write $\langle
\, , \, \rangle$ for the evaluation map $A^{*} \times A \to R$. 
\begin{defi}
  \label{def:3}
  Let $R$ denote a commutative ring. A $R$-\emph{Frobenius algebra}  is a
  quadruple $(A, \, e, \, \mu, \, \Phi)$ consisting of a finitely
  generated, projective, associative $R$-algebra $A$ with unit $e$,
  multiplication map $\mu: A \times A \to A$ and a left $A$-module
  isomorphism
  \begin{equation}
    \label{eq:3}
    \Phi : A \to A^{*}.
  \end{equation}
\end{defi}
In the case when $R = \comp$ we obtain the following equivalent definition.
\begin{defi}
  \label{def:4}
  A \emph{Frobenius algebra} is a quadruple $(A, \, e, \, \mu, \,
  \theta )$ consisting of a finitely generated, associative
  $\comp$-algebra $A$ with unit $e$,  multiplication map $\mu: A
  \times A \to A$ and a linear form
  \begin{equation}
    \label{eq:3.1.05}
    \theta :A \rightarrow \comp,
  \end{equation}
  called \emph{the trace}, such that the bilinear form $g:A\times
  A\rightarrow \comp$ given by:
  \begin{equation}
    \label{eq:3.1.10}
    g(x,y)=\theta (\mu (x, y))
  \end{equation}
  is non-degenerate.
\end{defi}
The trace $\theta$ is related to $\Phi$ via the identity
$
  \theta(a) = \langle \Phi(a), e \rangle.
$
\begin{obs}\label{g_and_theta}
  The form $g$ will be called the \emph{metric}. The trace $\theta$
  can be recovered from $g$ and the unit element $e$ by
  $
  \theta(x) = g(x, e).
  $
\end{obs}
We shall usually write $(x,y) \to xy $ for the multiplication map
$\mu(x,y)$. An important piece of information associated to a
Frobenius algebras is the trilinear map
\begin{equation}
  \label{eq:3.1.15}
  c: V\times V\times V \to \comp, \; \text{given by} \;
  c(x,y,z)=\theta (xyz).
\end{equation}
\begin{defi}
  A $\comp$-Frobenius algebra $A$ is \emph{semisimple} if it has no
  nilpotents.
\end{defi}
An important result in the theory of semisimple Frobenius algebras
over $\comp$ is given by the following proposition, see~\cite[Prop
2.2]{hitchin97:_froben_manif} for a proof.
\begin{proposition}\label{id-basis}
  If $A$ is a semisimple commutative $\comp$-Frobenius algebra of
  dimension $n$, then there exists a basis $e_{1}, \dots, e_{n}$ of
  $A$ such that:
  \begin{enumerate}
  \item $e_{i}^{2} = e_{i}, \; i = 1, \dots , n$
  \item $e_{i} e_{j} = 0, \; i,j = 1, \dots, n, \; i \ne j$
  \item $\dim_{\comp} \, e_{i} A = 1,   \; i = 1, \dots , n$.
  \end{enumerate}
  The basis is unique up to permutations of the elements.
\end{proposition}

\medskip

\subsection{Calabi-Yau and Cardy categories}
\label{sec:calabi-yau-categ}

Let $R$ be a commutative ring. 
\begin{defi}\label{def:cy0}
  A Calabi-Yau category over $R$ is a category $\scr{C}$
  satisfying:
\begin{enumerate}
\item For any pair of objects $a, b$ of $\scr{C}$ the space of
  homomorphisms 
  \begin{displaymath}
    E_{ab} := \mathrm{Hom}_{\scr{C}} (a, b)
  \end{displaymath}
  is a finitely generated, projective $R$ module.
\item The composition
  \begin{displaymath}
    E_{ab} \times E_{bc} \to E_{ac}
  \end{displaymath}
  is an $R$-bilinear map.
\item For each $a \in \mathrm{Obj} (\scr{C})$ there exists a
  homomorphism of $R$-modules
  \begin{equation}
    \label{eq:12.1}
    \theta_{a} : E_{a a} \to R, 
  \end{equation}
  that induces a left $E_{aa}$-modules isomorphism
  \begin{equation}
    \label{eq:7}
    E_{aa} \to E_{aa}^{*},
  \end{equation}
  where $E_{aa}^{*}$ denotes the dual $R$-module. This condition
  implies that $E_{aa}$ is a Frobenius algebra over $R$. 
\item The pairings
    \begin{eqnarray}
      \label{eq:12.2}
      E_{ab} \otimes E_{ba} \to E_{aa} \xrightarrow{\theta_{a}} R \\
      E_{ba} \otimes E_{ab} \to E_{bb} \xrightarrow{\theta_{b}} R
    \end{eqnarray}
    induce isomorphisms $E_{ab} \simeq E_{ba}^{*}$.  If $\varphi \in E_{ab}$, and
    $\psi \in E_{ba}$, then
    \begin{equation}
      \label{eq:12.3}
      \theta_{a} (\varphi \cdot \psi) = \theta_{b} (\psi \cdot \varphi)
    \end{equation}
\end{enumerate}
\end{defi}
\begin{obs}
  This definition is an extension of the usual definition of
  Calabi-Yau
  category~\cite{costello:07}. 
  The main change is that we replace
  $\comp$-vector spaces by finitely generated, projective $R$
  modules. This is essentially the same change from the definition of
  Frobenius algebra over a field to Frobenius algebras over a ring. 
\end{obs}
\begin{defi}\label{def:cy1}
  Let $A$ be a commutative Frobenius algebra over $R$.  A \emph{Calabi-Yau
  category} over $A$ is a category $\scr{C}$ satisfying:
\begin{enumerate}
\item For each object $a$ there exists a pair of  $R$-linear morphisms 
  \begin{equation}
    \label{eq:12.4}
    \imath_{a}:A \to E_{aa}, \,  \, \text{and} \, \, \imath^{a}: E_{aa} \to
    A. 
  \end{equation}
  such that
  \begin{enumerate}
  \item  $\imath_{a}$ is a homomorphism of $R$-algebras.
  \item For $r \in A$ and $\psi \in E_{ab}$ it holds that
    \begin{equation}
      \label{eq:12.6}
      \imath_{a}(r) \psi = \psi \imath_{b}(r).
    \end{equation}
  \item The morphisms $\imath_{a}$ and $\imath^{a}$ are adjoints in
    the sense that
    \begin{displaymath}
      \theta (\imath^{a} (\psi) \phi) = \theta_{a}(\psi
      \imath_{a}(\phi)), 
    \end{displaymath}
    for all $\psi \in E_{aa}$, $\phi \in A$.
  \end{enumerate}
\end{enumerate}
\end{defi}
Since $E_{ab}$ and $E_{ba}$ are in duality, if $E_{ab}$ is a free
$R$-module and $\psi_{\nu}$ is a basis of $E_{ab}$ let $\psi^{\nu}$ be
the dual basis of $E_{ba}$. Define $\pi_{b}^{a}: E_{aa} \to E_{bb}$ by
\begin{equation}
  \label{eq:4}
  \pi_{b}^{a} (\psi) := \sum_{\nu} \, \psi_{\nu} \psi \psi^{\nu}.
\end{equation}
\begin{proposition}
  \label{prop:1}
  If $E_{ab}$ (and $E_{ba}$) are free $R$-modules, then the
  homomorphism $\pi_{b}^{a}$ is independent of the choice of basis
  $\psi_{\nu}$ and $\psi^{\nu}$.
\end{proposition}
\begin{proof}
  Let $\varphi_{\alpha}$ and $\varphi^{\alpha}$ be another pair of
  dual basis of $E_{ab}$ and $E_{ba}$. Then
  \begin{displaymath}
    \varphi_{\alpha} = \sum_{\nu} \, a_{\alpha}^{\nu} \psi_{\nu},
    \quad \text{and} \quad \varphi^{\beta} = \sum_{\mu} \,
    b^{\beta}_{\mu} \psi^{\mu},
  \end{displaymath}
  for certain matrices $ [a_{\alpha}^{\nu}] $ and
  $[b^{\beta}_{\mu}]$. The conditions $\delta_{\alpha}^{\beta} =
  \varphi_{\alpha} \varphi^{\beta}$ and $\delta_{\nu}^{\mu} =
  \psi_{\nu} \psi^{\mu}$ imply that
  \begin{eqnarray*}
    \delta_{\alpha}^{\beta} 
    &=&  \varphi_{\alpha} \varphi^{\beta} =
    \left( \sum_{\nu} \, a_{\alpha}^{\nu} \psi_{\nu} \right) \left(\sum_{\mu} \,
    b^{\beta}_{\mu} \psi^{\mu} \right) \\
    &=& \sum_{\nu} \sum_{\mu} a_{\alpha}^{\nu} b^{\beta}_{\mu}
    \psi_{\nu} \psi^{\mu} = \sum_{\nu} \sum_{\mu} a_{\alpha}^{\nu}
    b^{\beta}_{\mu} \delta_{\nu}^{\mu} \\
    &=& \sum_{\nu} a_{\alpha}^{\nu} b^{\beta}_{\mu}. 
  \end{eqnarray*}
  Hence the matrix $[a_{\alpha}^{\nu}] $ is the inverse of the matrix
  $[b^{\beta}_{\mu}]$. Therefore if $\psi \in E_{aa}$
  \begin{eqnarray*}
    \sum_{\alpha} \varphi_{\alpha} \psi \varphi^{\alpha} &=&
    \sum_{\alpha, \nu. \mu} a_{\alpha}^{\nu} b^{\alpha}_{\mu}
    \psi_{\nu} \psi \psi^{\mu} \\
    &=&  \sum_{\nu. \mu} \delta_{\mu}^{\nu}  \psi_{\nu} \psi
    \psi^{\mu} = \sum_{\nu} \psi_{\nu} \psi
    \psi^{\nu}.
  \end{eqnarray*}
\end{proof}
We want to extend the definition of $\pi_{b}^{a}$  to finitely generated,
projective $R$-modules. Let $X = \mathrm{spec} (R)$. Then the modules
$E_{ab}$ and $E_{ba}$ define locally free sheaves $\tilde{E}_{ab}$ and
$\tilde{E}_{ba}$ over $X$. Let $U_{i}, \, i \in I$ be a covering of
$X$ such that $\tilde{E}_{ab} (U_{i})$ and $\tilde{E}_{ba} (U_{i})$
are $\mathcal{O}_{X} (U_{i})$-free modules. 
Then for each pair of
indices $i, j \in I$ also the modules $\tilde{E}_{ab} (U_{i} \cap
U_{j})$ and $\tilde{E}_{ba} (U_{i} \cap U_{j})$ are free $\mathcal{O}_{X} (U_{i}
\cap U_{j})$ modules.

Hence we have homomorphisms
\begin{eqnarray*}
  \pi_{b}^{a} (U_{i}) &:& \tilde{E}_{aa} (U_{i})  \to \tilde{E}_{bb}
  (U_{i}), \\
  \pi_{b}^{a} (U_{j}) &:& \tilde{E}_{aa} (U_{j})  \to \tilde{E}_{bb}
  (U_{j}), \\
  \pi_{b}^{a} (U_{i}\cap U_{j}) &:& \tilde{E}_{aa} (U_{i} \cap U_{j} )
  \to \tilde{E}_{bb}   (U_{i} \cap U_{j}).
\end{eqnarray*}
By Proposition~\ref{prop:1} the restriction of $\pi_{b}^{a} (U_{i})$
and $\pi_{b}^{a} (U_{j})$ to $U_{i} \cap U_{j}$ coincide with
$\pi_{b}^{a} (U_{i}\cap U_{j})$. Hence there is a globally well defined
homomorphism of sheaves
$
  \pi_{b}^{a}: \tilde{E}_{aa} (X) \to \tilde{E}_{bb} (X)
$
which is the same as a module homomorphism
\begin{displaymath}
  \pi_{b}^{a}: E_{aa} \to E_{bb}.
\end{displaymath}
\begin{defi}
  \label{def:5}
  A \emph{Cardy} category is a Calabi-Yau category over a
  $R$-Frobenius algebra $A$ that satisfy the following condition,
  called the \emph{Cardy condition},
  \begin{equation}
    \label{eq:5}
    \pi_{b}^{a}  = \imath_{b} \circ \imath^{a}.
  \end{equation}
  for any pair of labels $a, b$. 
\end{defi}

\medskip

\subsection{The maximal category of branes}
\label{subsec_boundary_semisimple}

In this section we will discuss some results of G. Moore and G. Segal
\cite{moore_segal1} regarding the structure of the algebras $E_{ab}$
corresponding to the open sector of an open and closed  topological field
theory. We will only consider the case for which the Frobenius algebra
$A$ of the closed sector is semisimple which is the hypothesis used by
Moore and Segal.

Let $A$ be an associative, commutative, semisimple Frobenius algebra
over $\comp$, and supppose $\dim_\comp A=n$. We then have a system of
orthogonal idempotents $e_1,\dots ,e_n$ which determine the simple
components; i.e.
$
A\cong \bigoplus_i\comp e_i,
$
and each summand $\comp e_i$ is isomorphic to $\comp$. The prime ideals
of $A$ can be identified with the set $X = \{e_{1}, \dots, e_{n}
\}$. This set plays the role of space-time. The algebra $A$ is the
algebra of observables. The Frobenius structure on $A$ induces a
measure $\mu$ on $X$ by $\mu(e_{i}) = \theta(e_{i})$.  

\begin{theorem}[\cite{moore_segal1}, Theorem 2]\label{theorem_2}
For each object $a\in \scr{B}$, the algebra $E_{aa}$ is semisimple.
\end{theorem}
\begin{obs}\label{dimensions}
  By the previous result the algebra $E_{aa}$ can be regarded as a sum
  $\bigoplus_i M(a,i)$ of matrix algebras
  $M(a,i):=\opnm{M}_{d_{(a,i)}}(\comp )$. In other words, it is
  possible to find complex vector spaces $V_{a,i}$ such that
\begin{equation}\label{iso_theorem2_ms}
E_{aa}\cong \bigoplus_{i=1}^n\opnm{End}(V_{a,i}),
\end{equation}
where $\dim V_{a,i}=d(a,i)$. Moreover, the matrix algebra
$\opnm{M}(a,i)=\opnm{End}(V_{a,i})$ corresponds under the isomorphism
\eqref{iso_theorem2_ms} with the subalgebra
$\iota_a(e_i)E_{aa}$. Elements of $E_{aa}$ will be denoted by a tuple
$\sigma =(\sigma_1,\dots ,\sigma_n)$, where $\sigma_i\in M(a,i)$. If
$\varepsilon_i\in E_{aa}$ denotes the tuple consisting of the identity
matrix $1_{a,i}\in M(a,i)$ in the $i$-th coordinate and all others
equals to zero, then $\iota_a(e_i)=\varepsilon_i$ or is equal to zero.
\end{obs}

We can give an explicit characterization for the morphisms $\theta_a$,
$\iota^a$ and $\pi^a_b$. For $\sigma =(\sigma_1,\dots ,\sigma_n )\in
E_{aa}$, the equality $\theta_a(\sigma \tau)=\theta_a (\tau \sigma)$
implies that
$$\theta_a(\sigma )=\sum_i\lambda_i \opnm{tr}(\sigma_i)$$
for some constants $\lambda_i \in \comp$.

Fixing a square root $\lambda_i=\sqrt{\theta (e_i)}$ for each $i$, we
arrive at the following expressions
$$
\begin{aligned}
\theta_a (\sigma ) &= \sum_i\sqrt{\theta (e_i)} \opnm{tr}(\sigma_i), \\
\iota^a(\sigma )   &= \sum_i\frac{\opnm{tr}(\sigma_i)}{\sqrt{\theta (e_i)}}e_i, \\
\pi_b^a(\sigma )   &= \sum_i\frac{\opnm{tr}(\sigma_i)}{\sqrt{\theta (e_i)}}\iota_b(e_i), \\
\end{aligned}
$$
where in the last equality, the trace $\opnm{tr}$ is the one
corresponding to $E_{aa}$.

A characterization like the one provided in theorem \ref{theorem_2}
holds for the spaces $E_{ab}$.

\begin{lemma}[\cite{moore_segal1}]\label{ms_theorem2_bis}
  If $C$ is semisimple, then for each pair $a,b\in \scr{B}$ we have an
  isomorphism
\begin{equation}\label{semisimple_2bis}
E_{ab}\cong \bigoplus_{i=1}^n\operatorname{Hom}_{\comp }(V_{a,i},V_{b,i}),
\end{equation}
for some finite-dimensional complex vector spaces $V_{a,i},V_{b,i}$.
\end{lemma}

Note that the vector spaces in the right hand side of equation
\eqref{semisimple_2bis} are the ones appearing in the decompositions
of $E_{aa}$ and $E_{bb}$; see remark \ref{dimensions}.

% Hasta aqui 2014-07-10

\subsection{Frobenius manifolds}
\label{sec:frobenius-manifolds}

In this section we shall briefly review the definition of a Frobenius
Manifold in the sense of Manin~\cite[Definition
1.1.1]{manin99:_froben_manif_quant_cohom_and_modul_spaces}. We shall
write $M$ or $(M, \mathcal{O}_{M})$ for a manifold, where the word
manifold means a $\mr{C}^{\infty}$, real analytic or complex analytic
manifold, and $\mathcal{O}_{M}$ denotes the (complexified) structure
sheaf of $M$-- we shall assume that $\mathcal{O}_{M}$ is a sheaf of
$\comp$ algebras and that for every $x \in M$ the reduced field
$k_{x}$ is $\comp$.  We shall write $\mathscr{T}M$ for the
complexified tangent sheaf and $\mathscr{T}^{*}M$ for the complexified
cotangent
sheaf. The sections of $\mathscr{T}M$ acts as derivations on
$\mathcal{O}_{M}$.  If $(x^{1}, \dots, x^{n})$ is a system of
coordinates, then the $x^{i}$ determine vector fields $\partial_{i}$
such that
\begin{displaymath}
  df = \sum_{i=1}^{m+n} \, dx^{i} \partial_{i} f.
\end{displaymath}
The vector fields $\partial_{i}$ locally generate $\mathscr{T}M$. The
one forms $dx^{i}$ locally generate the cotangent sheaf.
\begin{defi}
  A manifold $M$ with \emph{multiplication on the tangent sheaf} is a triple
  $(\mathscr{T}M, \mu, e)$, where
  \begin{equation}
    \label{eq:3.2.0}
    \mu: \mathscr{T}M \otimes \mathscr{T}M \to \mathscr{T}M
  \end{equation}
  is an associative $\mathcal{O}_{M}$-bilinear map of sheaves and $e$
  is a global 
  vector field which is the unit for $\mu$. 
\end{defi}
If a manifold $M$ has a multiplication on the tangent sheaf, then the
deviation from a Poisson algebra structure on $\mathscr{T}M$ is given
by the following expression
\begin{equation}
  \label{eq:3.2.1}
  P_{x} (z, w) := [x, \mu(z, w)] - \mu([x, z], w) - \mu (z, [x,w]).
\end{equation}
\begin{defi}
  A manifold $M$ with multiplication on the tangent sheaf is an
  $F$-\emph{manifold}, see for
  example~\cite{hertling99:_weak_froben_manif},  if the 
  multiplication $\mu$ satisfies
  \begin{equation}
    \label{eq:3.2.2}
    P_{\mu(x,y)} (z, w) = \mu(x, P_{y} (z, w)) + \mu(y, P_{x} (z,w)).
  \end{equation}
\end{defi}
\begin{defi}
  An \emph{affine flat structure} on a manifold $M$ is a subsheaf
  $\mathscr{T}_{\mathcal{F}} M $ of the tangent sheaf $\mathscr{T}M$
  of linear spaces such that $\mathscr{T}M = \mathscr{T}_{\mathcal{F}}
  M \otimes_{k} \mathcal{O}_{M}$ and the tangent bracket of pairs of
  its sections vanish.  The elements of $\mathscr{T}_{\mathcal{F}} M $
  are called \emph{flat vector fields}.
\end{defi}
\begin{defi}
  A metric $g$ on a manifold $M$ is \emph{compatible} with an
  affine flat structure if $g(x, y)$ is constant for flat vector
  fields $x$, $y$.
\end{defi}
The next condition defines the compatibility between the metric and
the multiplication on the tangent sheaf.
\begin{defi}
  A metric on a manifold with multiplication on the tangent sheaf
  is \emph{invariant} if
  \begin{equation}
    \label{eq:3.2.3}
    g(\mu(x,y), z) = g(x, \mu(y,z)).
  \end{equation}
\end{defi}
\begin{obs}\label{gtotheta}
  An invariant metric determines a tensor $c$ given by $c(x,y,z) =
  g(\mu(x,y), z)$ and a morphism $\theta_{M} : \mathscr{T}M \to
  \mathcal{O}_{M}$ given by $\theta (x) = g(e, x)$ -- see
  \ref{g_and_theta}.
\end{obs}
\begin{defi}
  \label{def:6}
  Let $M$ be a manifold. A \emph{pre-Frobenius} structure on $M$ is a
  triple $(\mathscr{T}_{\mathcal{F}} M, g, c)$ formed by an affine
  flat structure $\mathscr{T}_{\mathcal{F}} M$ on $M$, a compatible
  metric $g$,  and an even symmetric tensor $$c: S^{3} (\mathscr{T} M)
  \to \mathcal{O}_{M}.$$  A manifold with a pre-Frobenius structure will
  be called a \emph{pre-Frobenius manifold}.
\end{defi}
A pre-Frobenius manifold has a multiplication $\mu$ on the tangent
sheaf given by
\begin{displaymath}
  \mathscr{T} M \otimes \mathscr{T} M \xrightarrow{c} \mathscr{T}^{*}
  M \xrightarrow{g} \mathscr{T}M.
\end{displaymath}
The metric in this case is invariant under the multiplication.
A \emph{local potential} $\Phi$ for $\mathscr{T} M$
is an even function such that for any flat local tangent fields $x, y,
z$,
\begin{equation}
  \label{eq:4999}
  c(x, y, z) = (xyz) \Phi 
\end{equation}
A pre-Frobenius manifold is called \emph{potential} if $c$ admits
everywhere a local potential.
\begin{defi}\label{def:8}
  A \emph{Frobenius manifold} is an associative, potential
  pre-Frobenius manifold.
\end{defi}
\begin{defi}\label{def:7}
  A Frobenius manifold is called \emph{semisimple}, if there is
  everywhere a local isomorphism of sheaves of algebras $
  \mathscr{T}_{M} \simeq \mathcal{O}_{M}^{n}.  $
\end{defi}

\medskip

\subsubsection{The spectral cover $S_{M}$.}
\label{sec:3.3.1}

Let $M$ be an $n$-dimensional $F$-manifold in an analytic
category. Let $S_{M}$ be the relative affine spectrum of the
$\mathcal{O}_{M}$-algebra $\mathscr{T} M$. The space $S_{M}$
is a manifold in the same class that $M$ and it is endowed with
two structure maps $\widetilde{\pi}: S_{M} \to M$ and $s:
\mathscr{T}M \to \widetilde{\pi}_{*}
(\mathcal{O}_{S_{M}})$. The morphism $s$ is an isomorphism of
sheaves. If $M$ is a semisimple manifold, then $\widetilde{\pi}$ is
\'etale~\cite[section8.1]{manin99:_froben_manif_quant_cohom_and_modul_spaces}.
When $M$ is fixed we shall write $S$ for $S_{M}$.

If $M$ is a Frobenius manifold, then there is a natural trace $\theta_{\tm}:
\mathcal{O}_{S} \to \mathcal{O}_{M}$ given by the composition
\begin{displaymath}
  \mathcal{O}_{S} \to \pi_{*} \mathcal{O}_{S} \xrightarrow{s^{-1}}
\mathscr{T}M \xrightarrow{\theta_{M}} \mathcal{O}_{M},
\end{displaymath}
where $\theta_{M}$ is the trace defined in Remark~\ref{gtotheta}. With
this structure the Frobenius algebras obtained from $(\pi_{*}
(\mathcal{O}_{\widetilde{M}}), \theta_{\tm})$ are isomorphic, via $s$,
to the Frobenius algebras obtained from $(\mathscr{T} M, \theta_{M})$.

The construction of the spectral cover is part of a more general
framework, namely that of the \emph{analytic spectrum}, introduced by
C. Houzel \cite{houzel_gal2} to study finite morphism of analytic
spaces. He defines the analytic spectrum for algebras of finite
presentation over an analytic space, which include finite algebras
(those algebras which are coherent modules): let $\Gamma$ be a finite
presentation $\scr{O}_M$-algebra and $f:N\to M$ a space over $M$ (in
particular, if $E$ is a vector bundle, then its sheaf of sections is
coherent and thus of finite presentation). Define a contravariant
functor $S_\Gamma$ from spaces over $M$ to the category of sets by
\begin{displaymath}
  S_\Gamma (N,f)=\operatorname{Hom}_{\scr{O}_N\text{-}{\rm alg}}(f^*\Gamma ,\scr{O}_N)
\end{displaymath}
This functor is then representable, and we have a bijection between
$S_\Gamma (N,f)$ and holomorphic maps $N\to \operatorname{Specan}
\Gamma$, where $\operatorname{Specan}\Gamma$ is the analytic
spectrum. Even with these nice algebras, the space
$\operatorname{Specan}\Gamma$ may have singularities. For detailed
descriptions we refer the reader to \cite{houzel_gal2}; check also
\cite{fischer:_cng}. The case in which we are interested deals with a
bundle of algebras $E$ such that $E_x$ is semisimple for each $x$ (see
below). If $M=N$ and $f:M\to M$, then the construction of the analytic
spectrum provides a bijection between the subspace of the dual bundle
$(f^*E)^*$ consisting of morphisms of algebras and maps $M\to
\operatorname{Specan}\Gamma_E$. For
$f=\text{id}_M$, this is just expressing that every morphism of
algebras $\varphi :E\to \comp$ is determined by a map $M\to
\operatorname{Specan}\Gamma_E$ (for each $x$ this is just choosing the
kernel of the restriction $\varphi_x:E_x\to \comp$).

\begin{proposition}\label{isom_1}
  For a bundle of algebras $E$ over $M$ there exists an isomorphism of
  $\mathscr{O}_M$-algebras
  \begin{equation}\label{iso_2}
    \pi_*\scr{O}_{S_E}\cong \Gamma_E,
  \end{equation}
\end{proposition}
\begin{proof}
consider the sequence of maps
\begin{displaymath}
  \Gamma_E\longrightarrow p_*\mathscr{O}_{E^*}\longrightarrow \pi_*\mathscr{O}_{S_E},
\end{displaymath}
\begin{displaymath}
  X\longmapsto \widetilde{X}\longmapsto \widetilde{X}|_S
\end{displaymath}
where $p:E^*\rightarrow M$ is the canonical projection (we are
considering $S_E$ as a subspace of $E^*$; then $\pi$ is just the
restriction of $p$ to $S_E$), and
$\widetilde{X}:p^{-1}(U)=E^*|_U\rightarrow \comp$ is the map given by
$$\widetilde{X}(x,\varphi )=\varphi (X(x)).$$
The composite map
\begin{equation}\label{iso}
\Gamma_E\longrightarrow \pi_*\scr{O}_{S_E}
\end{equation}
is then easily seen to be an isomorphism of $\mathscr{O}_M$-algebras
(recall that $(x,\varphi )\in S_E$ if and only if $\varphi$ is an
algebra homomorphism).

The inverse can be described easily: Given a map
$\widetilde{f}:\pi^{-1}(U)\rightarrow \comp$, let
$X_{\widetilde{f}}\in \Gamma_E(U)$ be the local section defined as
follows: pick an $x\in U$ an assume that $U$ is semisimple (if it is
not, we can choose a smaller open neighborhood around $x$); let
$\{e_i\}$ be a local frame of idempotent sections for $E|_U$. Then
\begin{displaymath}
  X_{\widetilde{f}}(x)=\sum_i\widetilde{f}(x,\varphi_i)e_i(x),
\end{displaymath}
where $\varphi_i:E_x\rightarrow \comp$ is the algebra homomorphism
which verifies $\varphi_i(e_i(x))\neq 0$ (in fact, $\varphi_i
(e_i(x))=1$ as $\varphi_i(1)=1$). The assignment $\widetilde{f}\mapsto
X_{\widetilde{f}}$ is then the inverse of \eqref{iso}.
\end{proof}

Moreover, each summand $\scr{O}_{S,y}\otimes_{\scr{O}_{x_0}}\comp$ is
invariant under ths action of any multiplication operator, and thus it
is the space of generalized eigenvectors.

It holds the following result, which is in fact Housel's definition of
the spectral cover.

\begin{proposition}
Let $E\to M$ be a bundle of associative and commutative algebras. Then
\begin{enumerate}
\item The analytic spectrum $S_E$ represents the functor (which we
  denote with the same symbol) $S_E
  (N,f)=\operatorname{Hom}_{\scr{O}_N-\text{alg}}(f^*E,\comp)$ from
  spaces over $M$ to the category of sets (here $\comp$ means the
  trivial line bundle $N\times \comp$).
\item If $E_x$ is semisimple for each $x$, then $\pi :S_E\to M$ is a
  covering space.
\end{enumerate}
\end{proposition}

\medskip

\section{2-Vector spaces and 2-vector bundles}
\label{sec:2-vector-spaces}

\subsection{2-Vector Spaces}
\label{sec:2-vector-spaces-1}

We will now give an overview of the categorical analogues of vector
spaces and vector bundles. There are several definitions of 2-vector
space in the literature due, among others, to Kapranov-Voevodsky
\cite{kn:kv}, Baez-Crans \cite{baezcrans:2vector} and Elgueta
\cite{elgueta:2vector}. We will adopt the definition of 2-vector
spaces of Kapranov and Voevodsky. The references for our treatment of
monoidal categories are~\cite{maclane:_catwm, kelly:_enriched} 
 \begin{defi}
   \label{defi:1}
   A \emph{rig category} is a category ${\bf R}$ with two symmetric
   monoidal structures $({\bf R},\oplus ,{\bf 0})$ and $({\bf
     R},\otimes ,{\bf 1})$ together with distributivity natural
   isomorphisms
   $
     X\otimes (Y\oplus Z)\longrightarrow (X\otimes Y)\oplus (X\otimes
     Z) 
    $  
   and
   $
     (X\oplus Y)\otimes Z\longrightarrow (X\otimes Z)\oplus (Y\otimes Z)
   $
   verifying some coherence axioms which are detailed in
   \cite{laplaza:_coh1, kelly:_coh2}.
\end{defi}

An important example, which will be extensively used in what follows,
is the category $\tsf{Vect}$ of finite dimensional vector spaces over
$\comp$. The operations are given by direct sum (with ${\bf 0}=\{0\}$,
the trivial vector space) and tensor product (with ${\bf 1}\cong
\comp$). 
\begin{defi}
  \label{defi:2}
  Let ${\bf R}$ be a rig category. A \emph{left module category} over
  ${\bf R}$ is a monoidal category $({\bf M}, \oplus , {\bf 0})$
  together with an action (bifunctor)
  $$
  \otimes :{\bf R}\times {\bf M}\longrightarrow {\bf M}
  $$
  and natural isomorphisms
  \begin{displaymath}
    \begin{aligned}
    A\otimes (B\otimes X) &\longrightarrow (A\otimes B)\otimes
    X \\
    (A\oplus B)\otimes X &\longrightarrow (A\otimes X)\oplus
    (B\otimes X) \\
    A\otimes (X\oplus Y) &\longrightarrow (A\otimes X)\oplus
    (A\otimes Y) \\
  \end{aligned}
  \end{displaymath}
  \begin{displaymath}
    \tau_X = \tau :{\bf 1}\otimes X\longrightarrow X 
    \qquad \rho_A=\rho:A\otimes {\bf 0}
    \longrightarrow {\bf 0} \qquad \lambda_X=\lambda
    :{\bf 0}\otimes X\longrightarrow {\bf 0}
  \end{displaymath}
for any given objects
$A,B\in {\bf R}$ and $X,Y\in {\bf M}$, which are required to satisfy
coherence conditions analogous to the ones for a rig
category. \emph{Right module categories} are defined analogously.
\end{defi}

An ${\bf R}$-module functor between ${\bf R}$-modules ${\bf M}$ and
${\bf N}$ is a functor $F:{\bf M}\to {\bf N}$ such that 
\begin{eqnarray*}
  F(X\oplus Y)&\cong& F(X)\oplus F(Y) \\
  F(A\otimes X)&\cong& A\otimes F(X).
\end{eqnarray*}
The isomorphisms should be natural in $X$ and $Y$ in the first case
and natural in $A$ and $X$ in the second case.

Given $n\in \nat$, consider now the product category $\tsf{Vect}^n$;
its objects and morphisms are $n$-tuples of vector spaces and linear
transformations respectively. The $\tsf{Vect}$ module structure is
provided by the operations
\begin{displaymath}
  \begin{aligned}
  (V_1,\dots ,V_n)\oplus (W_1,\dots ,W_n) &= (V_1\oplus W_1,\dots
  ,V_n\oplus W_n), \\ 
  V\otimes (V_1,\dots ,V_n) &= (V\otimes V_1,\dots ,V\otimes V_n).\\
\end{aligned}
\end{displaymath}
Any object $(V_1,\dots ,V_n)$ can be decomposed, just like vectors in
euclidean $n$-space, in the following way
$$
(V_1,\dots ,V_n)=(V_1\otimes \comp_1)\oplus 
\cdots \oplus (V_n\otimes \comp_n),
$$
where $\comp_i$ is the tuple whose $i$-th entry is equal to $\comp$
and all others equal to the trivial vector space. Using this
decomposition any $\tsf{Vect}$-module functor can be determined on
objects by its values in each $\comp_i$,
\begin{equation}\label{module_functor}
  F(V_1,\dots ,V_n)\cong (V_1\otimes F(\comp_1))\oplus \cdots \oplus (V_n\otimes F(\comp_n)).
\end{equation}

We can define some extra structure in the category of
$\mathbf{R}$-modules by introducing morphisms between morphisms or
2-arrows. Given two ${\bf R}$-modules ${\bf M}$ and ${\bf N}$ and
module functors $F,G:{\bf M}\to {\bf N}$, we define a 2-morphism
$\theta :F\to G$ as a natural transformation. This provides the
category of ${\bf R}$-modules with a structure of 2-category.

\begin{defi}
  \label{2-vs-defi}
  A $\tsf{Vect}$-module category ${\bf V}$ is called a \emph{2-vector
    space} if it is $\tsf{Vect}$-module equivalent to the
  product $\tsf{Vect}^n$ for some natural number $n$. In other words,
  ${\bf V}$ is a 2-vector space if and only if there exists a natural
  number $n$ and a $\tsf{Vect}$-module functor ${\bf V}\to
  \tsf{Vect}^n$ which is also an equivalence of categories.
\end{defi}

The proof of the following theorem can be found in \cite{kn:kv}.
\begin{theorem}\label{prop-10}
  If $F: \tsf{Vect}^n \rightarrow \tsf{Vect}^m$ is an equivalence,
  then $n =m$.
\end{theorem}

By the previous result, the number $n$ in definition~\ref{2-vs-defi}
is well defined and it is called the \emph{rank} of the 2-vector space
${\bf V}$. The 2-vector space $\tsf{Vect}^n$ plays, in this
categorical setting, the same role that the space $\comp^n$ plays in
linear algebra.  We will denote by $2\tsf{Vect}$ the (2-)category of
2-vector spaces of finite rank. Morphisms between 2-vector spaces can
be characterised in a similar way as linear maps between vector
spaces. To see this, consider first an $m\times n$ matrix
\begin{displaymath}
  A=\begin{pmatrix}
  V_{11} & \cdots & V_{1n} \\
  \vdots & \ddots & \vdots \\
  V_{m1} & \cdots & V_{mn} \\
\end{pmatrix}.
\end{displaymath}
where the entries $V_{ij}$ are $\comp$ vector spaces of finite
dimension.  If $V:=(V_1,\dots ,V_n) \in
\tsf{Vect}^n$, then the product
\begin{displaymath}
  AV=\left (\sum_jV_{1j}\otimes V_j,\dots ,\sum_jV_{mj}\otimes V_j\right )
\end{displaymath}
is a well defined object of the category $\tsf{Vect}^m$; given now a
map $f:=(f_1,\dots ,f_n):V\to W$, where $W:=(W_1,\dots ,W_n)$, there
exists an induced map $Af:AV\to AW$ given by
\begin{displaymath}
  Af=\left (\sum_j\opnm{id}_{1j}\otimes 
    f_j,\dots ,\sum_j\opnm{id}_{mj}\otimes f_j\right ),
\end{displaymath}
where $\opnm{id}_{ij}:V_{ij}\to V_{ij}$ is the identity map. Moreover,
the correspondence
\begin{displaymath}
  \begin{aligned}
  V &\mapsto AV \\
  f &\mapsto Af \\
\end{aligned}
\end{displaymath}
is a $\tsf{Vect}$-module functor $\tsf{Vect}^n\to \tsf{Vect}^m$.
Composition of such morphisms is given by usual multiplication of
matrices, and two matrices $A=(V_{ij})$ and $B=(W_{ij})$ of the same
size are naturally isomorphic if and only if $V_{ij}$ is isomorphic to
$W_{ij}$ for each $i,j$.

Note that equation \eqref{module_functor} readily implies that a
morphism $F:\tsf{Vect}^n\to \tsf{Vect}^m$ is naturally isomorphic to
the $m\times n$ matrix with columns given by $F(\comp_1),\dots
,F(\comp_n)$. For a morphism $F:{\bf V}\to {\bf W}$ between 2-vector
spaces, if $u:{\bf V}\to \tsf{Vect}^n$ and $v:{\bf W}\to \tsf{Vect}^m$
are equivalences with inverses $\widetilde{u}$ and $\widetilde{v}$
respectively, then $vF\widetilde{u}$ is naturally isomorphic to a
matrix $A$, and hence $F$ can be represented as $\widetilde{v}Au$ for
some matrix $A$.

Let now $A=(V_{ij})$ be an $n\times n$ matrix which is an equivalence
$\tsf{Vect}^n\to \tsf{Vect}^n$, and let $B=(W_{ij})$ represent the
inverse, up to equivalence. As the identity morphism of $\tsf{Vect}^n$
can be represented by the ``scalar'' matrix $\comp \opnm{Id}$, we have
natural isomorphisms $AB\cong \comp \opnm{Id} \cong BA$. Taking
dimensions coordinatewise we can form the dimension matrices
$d(A):=(\dim V_{ij})$ and $d(B)=(\dim W_{ij})$. Then, as the dimension
matrices has natural entries, necessarily $\det d(A)=\pm 1$. But not
every matrix satisfying this property is in fact an equivalence, and
this is the main problem behind the short supply of equivalences
$\tsf{Vect}^n\to \tsf{Vect}^n$. For example, take $n=2$ and consider
the morphisms given by the matrices
\begin{displaymath}
  A_k=
  \begin{pmatrix}
    \comp & \comp \\
    \comp^{k-1} & \comp^k \\
  \end{pmatrix}.
\end{displaymath}
Then $d(A_k)=\left (\begin{smallmatrix} 1 & 1 \\ k-1 & k
    \\ \end{smallmatrix} \right )$ and $\det d(A_k)=1$. But, no matter
which $k\in \nat$ we choose, there is no inverse for $A_k$, and hence
it is not an equivalence of 2-vector spaces. The example below
explicitly shows the scarcity of equivalences for $n=2$.

\begin{ej}
  Let $A=(V_{ij})$ be an autoequivalence of $\tsf{Vect}^2$ and
  $B=(W_{ij})$ and inverse. Let $a_{ij}:=\dim V_{ij}$, $b_{ij}:=\dim
  W_{ij}$ and then $d(A)=(a_{ij})$ and $d(B)=(b_{ij})$. From the
  natural isomorphisms $AB\cong \comp \opnm{Id}\cong BA$ we deduce
  that the following equations must hold
  \begin{equation}\label{eq_ej_2}
    a_{i1}b_{1j}+a_{i2}b_{2j}=\delta_{ij},
  \end{equation}
  for $i,j=1,2$. In particular, the matrix $d(B)$ is the inverse of
  the matrix $d(A)$; hence
$$d(B)=\varepsilon \begin{pmatrix}
  a_{22} & -a_{12} \\
  -a_{21} & a_{11} \\
\end{pmatrix},
$$
where $\varepsilon =\pm 1$ is the determinant of $d(A)$. If
$\varepsilon =1$, then necessarily $a_{12}=a_{21}=0$; this fact
together with equation \eqref{eq_ej_2} yields
$$a_{ii}b_{ii}=1$$
for $i=1,2$, and then $a_{11}=a_{22}=1$. For $\varepsilon =-1$ we
obtain $a_{ii}=0$ for $i=1,2$ and $a_{12}=a_{21}=1$. Thus, the only
equivalences $\tsf{Vect}^2\to \tsf{Vect}^2$ (up to isomorphism) have
the form
$$\begin{pmatrix}
  \comp & 0 \\
  0 & \comp \\
\end{pmatrix}\quad , \quad
\begin{pmatrix}
  0 & \comp \\
  \comp & 0 \\
\end{pmatrix}.
$$
\end{ej}

% Hasta aqui 2014-07-21

%%%%%%%%%%%%%%%%%%%%%%%%%%%%%%%%
\subsubsection{2-Vector Bundles}

The notion of 2-vector bundle (of rank 1) was introduced by Brylinski
in \cite{brylinski:_catvb} as a way of describing some cohomology
classes associated to symplectic manifolds in terms of 2-vector spaces
(as an alternative to gerbes). His definition resembles the definition
of the sheaf of sections of a vector bundle. Another notion of
2-vector bundle was proposed by Baas, Dundas and Rognes (BDR) in
\cite{bdr:_2vb}.  Their definition resembles the cocycles for a vector
bundles.

We shall give here another definition of 2-vector bundles which is a
generalisation to higher ranks of Brylinski's definition. We shall
define some geometric 2-vector bundles naturally associated to a
Frobenius manifold.  In section~\ref{sec:maxim-categ-bran} we shall
show that the 2-vector bundles constructed in this way are 2-vector
bundles in the sense of Baas, Dundas and Rognes.

If ${\bf R}\to{\bf B}\leftarrow {\bf M}$ are fibred categories or
stacks over ${\bf B}$, then an action of ${\bf R}$ on ${\bf M}$ is a
morphism of fibred categories ${\bf R}\times_{{\bf B}}{\bf
  M}\longrightarrow {\bf M}$. If ${\bf M}$, has some extra structure
we shall ask the action to preserve such structure. For instance, if
the category ${\bf M}$ is additive, then we should have a natural
distributivity isomorphism $A\cdot (X\oplus Y)\cong A\cdot X\oplus
A\cdot Y$, plus other properties involving ${\bf 1}$ and ${\bf 0}$.

We shall write $[\tsf{Vect}, M]$ for the fibred category that
associated to each open set $U \subset M$ the category of vector
bundles over $U$.  The definition of 2-vector bundle given by
Brylinski in \cite{brylinski:_catvb} reads as follows.
\begin{defi}
  \label{defi:3}
  Let $M$ be a manifold and let $\tsf{Op}(M)$ denote the category of
  open sets of $M$. A fibred category ${\bf M}\to \tsf{Op}(M)$ is said
  to be a \emph{2-vector bundle of rank $1$ over $M$} if the following
  conditions hold:
  \begin{enumerate}
  \item For each open subset $U\subset M$, the fibre ${\bf M}(U)$ is
    an additive category.
  \item There exists an action $(E,X)\mapsto E\cdot X$ of the (fibred)
    category $[\tsf{Vect},M]$ on ${\bf M}$.
  \item Given any $x\in M$, there exists an open neighbourhood $U$of $x$
    and an object $X_U\in {\bf M}(U)$ (called a \emph{local
      generator}) such that the functor $\tsf{Vect}(U)\to {\bf M}(U)$
    given by $E\mapsto E\cdot X_U$ is an equivalence of categories,
    where $\cdot$ denotes the action.
  \item ${\bf M}\to \tsf{Op}(M)$ is a stack.
  \end{enumerate}
\end{defi}

We now extend the definition to higher ranks. For some technical
reasons instead of the (fibred) category of vector bundles, we shall
consider the (fibred) category of locally-free sheaves over $M$.  We
shall write $\tsf{LF}_{\scr{O}_M}$ for the (fibred) category of
locally free $\scr{O}_M$-modules on ${\bf M}$.
\begin{defi}\label{2bundle_n}
  A fibred category ${\bf M}\to \tsf{Op}(M)$ is said to be a
  \emph{2-vector bundle of rank $n$ over $M$} if and only if the
  following conditions hold:
  \begin{enumerate}
  \item For each open subset $U\subset M$, the fibre ${\bf M}(U)$ is
    an additive category.
  \item There exists an action $(\scr{M},X)\mapsto \scr{M}\cdot X$ of
    $\tsf{LF}_{\scr{O}_M}$ on $\mathbf{M}$.
  \item Given any $x\in M$, there exists an open neighbourhood $U\ni x$
    and objects $X_1,\dots X_n$ in ${\bf M}(U)$ (called \emph{local
      generators}) such that the functor $\tsf{LF}^n_{\scr{O}_U}\to
    {\bf M}(U)$ given by
    \begin{displaymath}
      (\scr{M}_1,\dots ,\scr{M}_n)
      \longmapsto \scr{M}_1\cdot X_1
      \oplus \cdots \oplus \scr{M}_n\cdot X_n       
    \end{displaymath}
is an equivalence of categories.
\item ${\bf M}\to \tsf{Op}(M)$ is a stack.
\end{enumerate}
\end{defi}
\begin{obs}
  Note that the local equivalence of the previous definition preserves
  both the action and the additive structure; that is, if $\Phi$ is
  such an equivalence, $\scr{L}\in \tsf{LF}_{\scr{O}_U}$ and
  $\scr{M},\scr{N} \in \tsf{LF}^n_{\scr{O}_U}$, then
  \begin{displaymath}
    \Phi \left ((\scr{L}\otimes \scr{M})\oplus \scr{N}\right )
    \cong \left (\scr{L}\otimes \Phi (\scr{M})\right )
    \oplus \Phi (\scr{N}).
  \end{displaymath}
\end{obs}
\begin{ej}
  Let $M=\{x\}$ be a one-point space. A 2-vector bundle of rank $n$
  over $M$ is then an additive category ${\bf M}$ equivalent to the
  category $\tsf{LF}^n_{\scr{O}}$. As $\scr{O}(M)\cong \comp$, then
  ${\bf M}$ is equivalent to the n-fold product of the category of
  $\comp$-modules; that is, it is a 2-vector space (of rank $n$).
\end{ej}
The following result shall be useful later.
\begin{proposition}
  Let $\Phi :\tsf{LF}^n_{\scr{O}_M}\to \tsf{LF}^m_{\scr{O}_M}$ be a
  functor which preserves the action and the additive structure. Then
  there exists an $m\times n$ matrix $A:=(\scr{M}_{ij})$ of
  $\scr{O}_M$-modules such that $\Phi$ is naturally isomorphic to
  multiplication by $A$.
\end{proposition}

The proof is completely analogous to the one for 2-vector
spaces. Moreover, this kind of morphisms share with 2-vector spaces
the same shortage of equivalences.

We shall now introduce Baas-Dundas-Rognes (BDR) 2-vector bundles -- 
see~\cite{bdr:_2vb} for more details.
\begin{defi}
  \label{bdr_2bundle}
  Let $A$ be a poset and let $\mathfrak{U}=\{U_{\alpha }\}_{\alpha \in
    A}$ be an ordered open cover of a topological space $M$ by open
  subsets. A \emph{Bass-Dundas-Rognes 2-vector bundle} (\emph{BDR
    2-vector bundle} for short) is a law that assigns to each pair
  $\alpha, \beta \in A$ a matrix $E^{\alpha \beta}:=(E_{ij}^{\alpha
    \beta})$ of (constant rank) vector bundles over $U_\alpha \cap
  U_\beta =U_{\alpha \beta}$ (for each $\alpha < \beta$) subject to
  the following conditions:
  \begin{enumerate}
  \item $\det \left (\opnm{rk}E_{ij}^{\alpha \beta}\right )=\pm 1$.
  \item For $\alpha <\beta <\gamma$ in $A$ and $U_{\alpha \beta \gamma
    }\neq \emptyset$, we have isomorphisms
    \begin{displaymath}
      \phi^{\alpha \beta \gamma}_{ik}:
      \bigoplus_jE_{ij}^{\alpha \beta}\otimes E_{jk}^{\beta \gamma}
      \stackrel{\cong}{\longrightarrow} E_{ik}^{\alpha \gamma}.
    \end{displaymath}
    As for morphisms of 2-vector spaces, this condition can also be
    expressed in matrix form $\phi^{\alpha \beta \gamma}:E^{\alpha
      \beta}E^{\beta \gamma }\cong E^{\alpha \gamma }$.
  \item For $\alpha <\beta <\gamma <\delta$ with $U_{\alpha \beta \gamma
      \delta}\neq \emptyset$, the following diagram of bundles over
    $U_{\alpha \beta \gamma \delta}$ should commute
    \begin{displaymath}
      \xymatrix{
        E^{\alpha \beta}\otimes (E^{\beta \gamma } \otimes E^{\gamma \delta }) 
        \ar[rr] \ar[d] & 
        & (E^{\alpha \beta } \otimes E^{\beta \gamma })\otimes E^{\gamma \delta } 
        \ar[d] \\
        E^{\alpha \beta }\otimes E^{\beta \delta } \ar[r] & E^{\alpha \delta } &
        E^{\alpha \gamma }\otimes  E^{\gamma \delta }, \ar[l] }
    \end{displaymath}
    where the top arrow is the associativity isomorphism derived from the
    associativity of the tensor product of vector bundles and the other
    arrows are defined from the isomorphisms of the previous item.
  \end{enumerate}
\end{defi}

\subsection{Azumaya algebras and twisted vector bundles}
\label{sec:twist-vect-bundl}

In this section we will introduce some basic material regarding
Azumaya algebras, as well as an introduction to twisted vector
bundles. The former are strongly related to the latter, and this
relationship will also appear later in section~\ref{sec:maxim-categ-bran}.  The
treatment of twisted bundles is mainly based on~\cite{karoubi:twisted_vector}.

\subsubsection{Azumaya Algebras}
\label{subsec_azumaya}

If $\mathbb{F}$ is a field (which we assume to have characteristic
equal to zero), a \emph{central simple algebra} over $\mathbb{F}$ is a
simple (associative) algebra with center equal to
$\mathbb{F}$. Replacing $\mathbb{F}$ with a commutative local ring $R$
leads to the notion of \emph{Azumaya algebra}; that is, an associative
$R$-algebra $A$ is an Azumaya algebra if and only if there exists some
$k\in \nat$ such that $A\cong R^k$ as $R$-modules and also the algebra
homomorphism $\varphi :A\otimes_RA^{\circ }\to \opnm{End}_R(A)\cong
\opnm{M}_k(A)$ given by $ \varphi (x\otimes y)(z)=xyz$ is an
isomorphism, where $A^{\circ}$ is the algebra with underlying set $A$
and operation given by $x\cdot y=yx$ (the right hand side is
multiplication in $A$). Auslander and Goldman
\cite{auslander_goldman} generalized this definition to include any
commutative (not necessarily local) base ring.
\begin{defi}\label{def:10}
  An \emph{Azumaya algebra} over $(M,\scr{O}_M)$ is a coherent sheaf
  of $\scr{O}_M$-algebras locally isomorphic to the sheaf
  $\operatorname{M}_k(\scr{O}_M)$.
\end{defi}
\begin{obs}
  \label{obs:az}
  By Proposition 2.1 (b) of~\cite{milne80:_etale_cohom}, see also
  section 1 of~\cite{grothendieck68:_le_group_de_brauer_i}, an Azumaya
  algebra over $(M,\scr{O}_M)$ is a locally free sheaf of algebras
  such that the fibres are $\operatorname{M}_k(\comp)$. 
\end{obs}

If $E \to M$ is a vector bundle, the sheaf of sections of 
$\mathrm{End}(E)$ is an Azumaya algebra. Not every Azumaya algebra has
this form. However the situation is different if one considers twisted
vector bundles.

\subsubsection{Twisted vector bundles}
\label{sec:twist-vect-bundl-1}

Twisted vector bundles can be thought of as a model for twisted
K-theory \cite{atiyah-segal:twisted_k}, just as vector bundles are
models of topological K-theory.
\begin{defi}
\label{def:11}
A \emph{twisted vector bundle} $\mathbb{E}$ over $M$ is a tuple
\begin{displaymath}
  \mathbb{E}=(\mathfrak{U},U_i\times V,g_{ij},\lambda_{ijk})
\end{displaymath}
consisting of the following data:
\begin{enumerate}
\item An open cover $\mathfrak{U}=\{U_i\}$ of $M$.
\item A (trivial) vector bundle $U_i\times V$ over each $U_i\in
  \mathfrak{U}$, where $V$ is a finite dimensional complex vector
  space (which shall usually be taken to be complex $n$-space).
\item Two families of maps $g_{ij}:U_{ij}\to \operatorname{GL}(V)$ and
  $\lambda_{ijk}\in \scr{O}(U_{ijk})$ such that $\lambda
  :=(\lambda_{ijk})$ is a $\check{\text{C}}$ech 2-cocycle, each map
  $\lambda_{ijk}$ takes values in $\comp^\times$ and
  \begin{displaymath}
    g_{ii}=1 \quad ,\quad g_{ji}=g_{ij}^{-1} \quad , 
    \quad g_{ij}g_{jk}=\lambda_{ijk}g_{ik} 
  \end{displaymath}
over $U_{ijk}$.
\end{enumerate}
\end{defi}

Two twisted bundles $\mathbb{E}=(\mathfrak{U},U_i\times
V,g_{ij},\lambda_{ijk})$ and $\mathbb{F}=(\mathfrak{V},V_r\times
V,f_{rs},\mu_{rst})$ will be regarded as equal if the cocycles of
$\mathbb{E}$ and $\mathbb{F}$ are equal over members of the refinement
$\mathfrak{U}\cap \mathfrak{V}$. 

Twisted vector bundles admit the
same operations as ordinary bundles.
\begin{defi}
  \label{def:9}
  Let $\mathbb{E}=(\mathfrak{U},U_i\times V,g_{ij},\lambda_{ijk})$ and
  $\mathbb{F}=(\mathfrak{U},U_i\times W,f_{ij},\mu_{ijk})$ be twisted
  vector bundles over $M$. A \emph{morphism} $\phi :\mathbb{E}\to
  \mathbb{F}$ is a family of bundle morphisms
  \begin{displaymath}
    \phi_i:U_i\times V\longrightarrow U_i\times W
  \end{displaymath}
  such that the following square
  \begin{equation}\label{diag_tvbm}
    \xymatrix{
      U_{ij}\times V \ar[r]^{\phi_j} 
      \ar[d]_{1\times g_{ij}} & U_{ij}\times W \ar[d]^{1\times f_{ij}} \\
      U_{ij}\times V \ar[r]_{\phi_i} & U_{ij}\times W}
  \end{equation}
  commutes.
\end{defi}
\begin{lemma}\label{isomorphic}
  Two twisted bundles $\mathbb{E}=(\mathfrak{U},U_i\times
  V,g_{ij},\lambda_{ijk})$ and $\mathbb{F}=(\mathfrak{U},U_i\times
  W,f_{ij},\mu_{ijk})$. are isomorphic if and only if there exists a
  family of maps $\{u_i:U_i\to \operatorname{Iso}(V,W)\}$ such that
$$f_{ij}=u_ig_{ij}u_j^{-1}.$$
\end{lemma}

Further properties are given in the following

\begin{lemma}\label{prop_tvb}
Let $\mathbb{E}$ and $\mathbb{F}$ be twisted bundles. Then
\begin{enumerate}
\item $\mathbb{E} \otimes \mathbb{F}\cong \mathbb{E}$ if and only if
  $\mathbb{F}$ is an ordinary line bundle.
\item The dual bundle $\mathbb{E}^{*}$ has twisting $\lambda^{-1}$.
\item If $\mathbb{E}$ has twisting $\lambda$ and $\mathbb{F}$ has
  twisting $\lambda^{-1}$, then $\mathbb{E}\otimes \mathbb{F}$ is an
  ordinary vector bundle. In particular, $\mathbb{E}^*\otimes
  \mathbb{F}$ is also a vector bundle if $\mathbb{E}$ and $\mathbb{F}$
  have the same twisting.
\item If $\mathbb{E}$ is defined over the trivial open cover
  $\mathfrak{U}=\{M\}$, then $\mathbb{E}$ is a trivial vector bundle,
  and conversely.
\end{enumerate}
\end{lemma}
Of particular interest is the twisted vector bundle
$\operatorname{Hom}(\mathbb{E},\mathbb{F})$, which is defined by
\begin{displaymath}
  \operatorname{Hom}(\mathbb{E},\mathbb{F})=
  (\mathfrak{U},U_i\times \operatorname{Hom}_\comp 
  (V,W),h_{ij},\lambda_{ijk}^{-1}\mu_{ijk}),
\end{displaymath}
where $h_{ij}:U_{ij}\to \operatorname{GL}(\operatorname{Hom}_{\comp
}(V,W))$ is given by $h_{ij}(x)(u)=f_{ij}(x)ug_{ij}(x)^{-1}$. If
$\mathbb{F}$ is also a $\lambda$-twisted bundle (i.e. $\mu =\lambda$),
then the data defining $\operatorname{Hom}(\mathbb{E},\mathbb{F})$ in
fact defines an ordinary vector bundle (there is no twisting!), which
is denoted by $\operatorname{HOM}(\mathbb{E},\mathbb{F})$. If
$\mathbb{E}=\mathbb{F}$, then
$\operatorname{HOM}(\mathbb{E},\mathbb{F})$ will be denoted
$\operatorname{END}(\mathbb{E})$.

\begin{lemma}[\cite{karoubi:twisted_vector}, Proposition 3.1]
  The vector space $\operatorname{Hom}_{\tsf{TVB}_\lambda
    (M)}(\mathbb{E},\mathbb{F})$ can be canonically identified with
  the space of sections of the bundle
  $\opnm{HOM}(\mathbb{E},\mathbb{F})$.
\end{lemma}

\begin{theorem}[\cite{karoubi:twisted_vector}, Theorem 3.2]
\label{tvb_azumaya}
Assume $A$ is an Azumaya algebra over $M$. Then, there exists a twisted bundle $\mathbb{E}$ such that
$$A\cong \operatorname{END}(\mathbb{E}).$$
\end{theorem}

\subsubsection{The Twisted Picard Group}
\label{sec:twisted-picard-group}

For the following discussion it will be useful to recall the
definition of the Picard group of a manifold $M$; consider the set of
isomorphism classes of (ordinary) line bundles over $M$. If $L,K$ are
line bundles, then $[L]\cdot [K]:=[L\otimes K]$ provides the set of
isomorphism classes of line bundles with a structure of abelian
group. This group is called the \emph{Picard group of $M$} and is
denoted by $\opnm{Pic}(M)$.

Analogously, twisted line bundles also enjoy some remarkable
properties, like line bundles do. Given a twisted bundle $\mathbb{E}$,
we shall denote by $[\mathbb{E}]$ its isomorphism class. Let us
restrict ourselves to considering isomorphism classes of twisted line
bundles over a manifold $M$. We define a product in the following way:
\begin{equation}\label{op_pic}
  [\mathbb{L}]\cdot [\mathbb{K}]:=[\mathbb{L}\otimes \mathbb{K}],
\end{equation}
extending the one for line bundles.

\begin{theorem}
  The set of isomorphism classes of twisted line bundles together with
  the operation \eqref{op_pic} is an abelian group which contains
  $\operatorname{Pic}(M)$ as a subgroup.
\end{theorem}
\begin{proof}
Associativity and commutativity of the operation follow from the ones of the tensor product, as stated in \ref{prop_tvb}.

Let $\mathbb{L}$ be a twisted line bundle; if $\epsilon^1$ denotes the trivial line bundle over $M$, then $\mathbb{L}\otimes \epsilon^1\cong \mathbb{L}$; to see this, consider the family of maps
$$\phi_i:U_i\times (\comp \otimes \comp )\longrightarrow U_i\times \comp$$
given by $\phi_i(x,z\otimes w)=(x,zw)$. These maps define a morphism
of twisted bundles
$$\phi :\mathbb{L}\otimes \epsilon^1\longrightarrow \mathbb{L},$$
with inverse given by the family $\phi_i^{-1}(x,z)=(x,z\otimes 1)$. Hence, $[\epsilon^1]=1$, the unit of the group.

Let now $[\mathbb{L}]$ be an arbitrary class. Then, $\mathbb{L}
\otimes \mathbb{L}^*$ is an ordinary line bundle; denoting this bundle
by $L$, we have that
\begin{displaymath}
  [\mathbb{L}]^{-1}=[\mathbb{L}^*\otimes L^*].
\end{displaymath}
The inclusion of $\opnm{Pic}(M)$ as a subgroup is clear from the
previous discussion.
\end{proof}

The group introduced in the previous theorem will be called the
\emph{twisted Picard group of $M$} and denoted by $\opnm{TPic}(M)$.

Assume now that $\opnm{TVB}(M)$ and $\opnm{Vect}(M)$ are sets
consisting of twisted bundles (with arbitrary twisting) over $M$ and
vector bundles over $M$, respectively, and consider the equivalence
relations $\mathbb{E}\sim \mathbb{E}\otimes \mathbb{L}$ and $E\sim
E\otimes L$, where $\mathbb{L}$ is a twisted line bundle and $L$ is a
line bundle. In the following result, $[\mathbb{E}]$ will denote the
class of $\mathbb{E}$ according to the relation $\mathbb{E}\sim
\mathbb{L}\otimes \mathbb{E}$; the same notation will be used for
ordinary vector bundles.

\begin{theorem}\label{bij_tensor}
There exists a non-canonical biyection
\begin{displaymath}
  \Psi :\opnm{TVB}(M)/_{\mathbb{E}\sim \mathbb{L}\otimes \mathbb{E}}\stackrel{\cong}
  {\longrightarrow}
  \opnm{Vect}(M)/_{E\sim L\otimes E}.
\end{displaymath}
\end{theorem}
\begin{proof}
  For each twisting $\lambda$, let us fix a twisted line bundle
  $\mathbb{L}_{\lambda}$ with that twisting. Now consider the map
  \begin{displaymath}
    \Psi [\mathbb{E}]=[\mathbb{E}\otimes \mathbb{L}_{\lambda^{-1}}],
  \end{displaymath}
  where $\mathbb{E}$ has twisting $\lambda$.

  We check that this correspondence is well-defined: first note that
  the twisting of $\mathbb{E}\otimes \mathbb{L}_{\lambda^{-1}}$ is
  $\lambda \lambda^{-1}=1$, and hence it is an ordinary line
  bundle. Now suppose that $[\mathbb{E}]=[\mathbb{F}]$, where
  $\mathbb{E}$ has twisting $\lambda$ and $\mathbb{F}$ twisting $\mu$;
  this implies the existence of a twisted line bundle $\mathbb{L}$
  such that $\mathbb{F}\cong \mathbb{L}\otimes \mathbb{E}$. In
  particular, if $\mathbb{L}$ has twisting cocycle equal to
  $\epsilon$, then $\mu =\epsilon \lambda$. We now have to check that
  $[\mathbb{E}\otimes \mathbb{L}_{\lambda^{-1}}]=[\mathbb{F}\otimes
  \mathbb{L}_{\mu^{-1}}]$; in other words, we should find a line
  bundle $L$ such that $\mathbb{E}\otimes
  \mathbb{L}_{\lambda^{-1}}\cong \mathbb{L}\otimes \mathbb{E}\otimes
  \mathbb{L}_{\mu^{-1}}\otimes L$. Take now
  \begin{displaymath}
    L:=\mathbb{L}_\mu \otimes \mathbb{L}^*\otimes \mathbb{L}_{\lambda^{-1}};
  \end{displaymath}
  then $L$ is an ordinary line bundle, as the twisting of the product
  of the right hand side is precisely $\mu
  \epsilon^{-1}\lambda^{-1}=\epsilon \lambda
  \epsilon^{-1}\lambda^{-1}=1$. We then have
  \begin{displaymath}
    \mathbb{L}\otimes \mathbb{E}\otimes \mathbb{L}_{\mu^{-1}}
    \otimes L\cong \mathbb{L}\otimes \mathbb{E}\otimes \mathbb{L}_{\mu^{-1}}
    \otimes \mathbb{L}_\mu \otimes \mathbb{L}^*\otimes 
    \mathbb{L}_{\lambda^{-1}}\cong \mathbb{E}\otimes \mathbb{L}_{\lambda^{-1}},
  \end{displaymath}
as desired.

Assume now that $\mathbb{E}$ and $\mathbb{F}$ are twisted bundles with
twistings $\lambda$ and $\mu$ respectively such that there exists a
line bundle $L_0$ with $\mathbb{F}\otimes \mathbb{L}_{\mu^{-1}}\cong
L_0\otimes \mathbb{E}\otimes \mathbb{L}_{\lambda^{-1}}$. Multiplying
by $\mathbb{L}_{\mu}$ at both sides, we obtain
\begin{displaymath}
  \mathbb{F}\otimes L_1\cong L_0\otimes \mathbb{E}\otimes 
  \mathbb{L}_{\lambda^{-1}}\otimes \mathbb{L}_{\mu},
\end{displaymath}
where $L_1=\mathbb{L}_{\mu}\otimes \mathbb{L}_{\mu^{-1}}$. Multiplying
now by the dual line bundle $L_1^*$ yields
\begin{displaymath}
  \mathbb{F}\cong \mathbb{E}\otimes \mathbb{L}_{\lambda^{-1}}
  \otimes \mathbb{L}_{\mu}\otimes L_0\otimes L_1^*.
\end{displaymath}
As $\mathbb{L}_{\lambda^{-1}}\otimes \mathbb{L}_{\mu}\otimes
L_0\otimes L_1^*$ is a twisted line bundle (with twisting $\mu
\lambda^{-1}$), then $[\mathbb{F}]=[\mathbb{E}]$ and hence $\Psi$ is
injective.

Let now $E$ be an arbitrary bundle. Then $E\otimes
\mathbb{L}_{\lambda}$ is a $\lambda$-twisted vector bundle and then
\begin{displaymath}
  \Psi [E\otimes \mathbb{L}_{\lambda}]=[E\otimes
\mathbb{L}_{\lambda}\otimes \mathbb{L}_{\lambda^{-1}}]=[E].
\end{displaymath}
\end{proof}

%%% Local Variables:
%%% mode: latex
%%% TeX-master: "master"
%%% End: 

\medskip

\section{Cardy fibrations}
\label{sec:calabi-yau-fibr}

We shall now define an extension of the notion of Calabi-Yau category.
\begin{defi}
  \label{def:13}
  Let $R$ be a commutative ring with unit. An $R$-linear category
  $\mr{C}$ is a \emph{Calabi-Yau category} if for each pair of objects
  $a, b$ in $\mr{C}$, the set of arrows
  $\operatorname{Hom}_{\mr{C}}(a,b)$ is a finitely generated
  projective $R$-module and for each element $a$ in $\mr{C}$ there
  exists a linear form
    \begin{displaymath}
      \theta_a:\operatorname{Hom}_{\mr{C}}(a,a)\longrightarrow R
    \end{displaymath}
    such that 
    \begin{enumerate}
    \item the induced pairing
    \begin{equation}\label{pair}
      \operatorname{Hom}_{\mr{C}}(a,b)\otimes_R
      \operatorname{Hom}_{\mr{C}}(b,a)\longrightarrow 
      \operatorname{Hom}_{\mr{C}}(a,a)\stackrel{\theta_a}{\longrightarrow }R
    \end{equation}
    is a perfect pairing
  \item given arbitrary arrows
    $\sigma :a\to b$ and $\tau :b\to a$, the equality $\theta_a(\tau
    \sigma )=\theta_b(\sigma \tau )$ holds. 
    \end{enumerate}    
\end{defi}

\begin{defi}
  \label{defi:5}
  Let $M$ be a semisimple manifold with multiplication $M$ a
  $CY$-category $\mr{B}$ over $M$ is a fibred category over the
  category of open sets of $M$ such that for each open set $U \subset
  M$ the category $\mr{B}(U)$ is an $\mr{O}(U)$-CY category.
\end{defi}

Let us fix a semisimple manifold with multiplication $M$, with
structure sheaf $\scr{O}=\scr{O}_M$ and let $\scr{B}$ be an
$\scr{O}$-linear CY category over $M$. For objects $a,b\in
\scr{B}(U)$, let us denote by $\Gamma_{ab}$ the presheaf
$\underline{\operatorname{Hom}}_U(a,b)$ over $U$ given by
\begin{equation}\label{sheaf_maps}
  V\longmapsto \operatorname{Hom}_{\scr{B}(V)}(a|_V,b|_V).
\end{equation}
By definition of CY category, we have that $\Gamma_{aa}$ is a
Frobenius $\scr{O}_U$-algebra for each $a\in
\scr{B}(U)$.

\begin{notation}
  Recall that if the base manifold is clear, we shall supress the
  subscript of the structure sheaf when taking local sections;
  e.g. instead of using the notation $\scr{O}_M(U)$ for $U\subset M$,
  we will only write $\scr{O}(U)$; and the restriction $\scr{O}_M|_U$
  shall be denoted $\scr{O}_U$. The same considerations are applied to
  the tangent sheaf $\scr{T}_M$ of a manifold $M$.
\end{notation}

We now turn to the relevant definitions.

\begin{defi}\label{cy_fibration}
  A \emph{Calabi-Yau (CY) fibration over a semisimple manifold $M$} is
  a pair $(\scr{B},\mathfrak{U})$ (the open cover shall be omitted
  form the notation), where $\scr{B}$ is a CY category over $M$ and
  $\mathfrak{U}=\{U_\alpha \}$ is an open cover of $M$, subject to the
  following conditions:
  \begin{enumerate}
  \item Each $U_\alpha \in \mathfrak{U}$ is semisimple.
  \item $\scr{B}$ is a stack.\footnote{In particular, the presheaf
      \eqref{sheaf_maps} is a sheaf.}
  \item Given any $U_\alpha\in \mathfrak{U}$ and objects $a,b\in
    \scr{B}(U_\alpha)$, the sheaf $\Gamma_{ab}$ is a locally-free
    locally finitely generated $\scr{O}_{U_\alpha}$-module. Objects of
    $\scr{B}(U)$ are called \emph{labels}, \emph{boundary conditions}
    or \emph{D-branes} over $U$.
  \item For each $U_\alpha \in \mathfrak{U}$ and each object $a\in
    \scr{B}(U_\alpha )$, we have transition (sheaf) homomorphisms
    \begin{displaymath}
      \iota_a:\scr{T}_U\longrightarrow \Gamma_{aa} 
      \quad , \quad \iota^a:\Gamma_{aa}\longrightarrow \scr{T}_U.     
    \end{displaymath}
The previous data is subject to the following conditions:

\begin{enumerate}
\item $\iota_a$ is a morphism of $\scr{O}_{U_\alpha}$-algebras
  (preserves multiplication and unit) and $\iota^a$ is an
  $\scr{O}_{U_\alpha}$-linear map. In particular, $\iota_a$
    provides $\Gamma_{aa}$ with a $\scr{T}_{U_\alpha}$-algebra
    structure.
\item $\iota_a$ is central: given $X\in \scr{T}(V)$ and $\sigma \in
  \Gamma_{ab}(V)$, we have
  \begin{equation}\label{centrality}
    \sigma \iota_a(X)=\iota_b(X)\sigma
  \end{equation}
  in $\Gamma_{ab}(V)$, for each $V\subset U_\alpha$.
\item There is an adjoint relation between $\iota_a$ and $\iota^a$
  given by
  \begin{equation}\label{adjoint}
    \theta (\iota^a(\sigma )X)=\theta_a(\sigma \iota_a(X)),
  \end{equation}
  for each $X\in \scr{T}_{U_\alpha}$ and $\sigma \in
  \Gamma_{aa}$.
\end{enumerate}
\end{enumerate}
\end{defi}
\begin{obs}
  For some technical considerations, we will assume that our CY
  fibrations $\scr{B}$ verify that for each open subset $U\subset M$,
  the skeleton $\opnm{sk}\scr{B}(U)$ of the category $\scr{B}(U)$ is a
  set.
\end{obs}

\subsection{Cardy Fibrations}

For $U_\alpha \in \mathfrak{U}$ open and $a,b\in \scr{B}(U_\alpha )$,
if $\Gamma_{ab}$ restricted to $U$ is trivial 
pick a local basis $\{\sigma_i\}$ of $\Gamma_{ab}$ and let
$\{\sigma^i\}$ be a basis of $\Gamma_{ab}^*$ dual to
$\{\sigma_i\}$. Define the map $\pi^a_b:\Gamma_{aa}\to \Gamma_{bb}$ by
$$\pi^a_b(\sigma )=\sum_i\sigma_i\sigma \sigma^i.$$
Some comments are in place: the sequence of maps
\begin{equation}\label{duality}
  \Gamma_{ba}\otimes \Gamma_{ab}\longrightarrow 
  \Gamma_{bb}\stackrel{\theta_a}{\longrightarrow}\scr{O}_U
\end{equation}
induces a duality isomorphism
$\Gamma_{ba}\stackrel{\cong}{\longrightarrow}\Gamma_{ab}^*$. The dual
basis in the definition of $\pi_b^a$ is in fact the preimage of the
dual basis of $\{\sigma_i\}$ under this isomorphism. Another key
observation is stated in the following

\begin{proposition}\label{cardy_well_def}
  The map $\pi^a_b$ does not depend on the chosen (local) basis.
\end{proposition}
\begin{proof}
  As $\Gamma_{aa}$, $\Gamma_{bb}$ and $\Gamma_{ba}$ are locally-free,
  we can pick an open cover $\mathfrak{U}_\alpha$ of $U_\alpha$ such
  that $\Gamma_{aa}|_V\cong \scr{O}^{n_{a}}$, $\Gamma_{ba}|_V\cong
  \scr{O}^{n_{ba}}$, etc. for each $V\in \mathfrak{U}_\alpha$. Pick
  then a basis $B=\{e_1,\dots ,e_{n_{ba}}\}$ for
  $\Gamma_{ba}|_V$.\footnote{By a \emph{basis} we mean a system of
    linearly independent generators $e_1,\dots ,e_{n_{ba}}\in
    \Gamma_{ba}(V)$ such that $\{e_1|_W,\dots ,e_{n_{ba}}|_W\}$ is
    also linearly independent and generates $\Gamma_{ba}(W)$ for each
    $W\subset V$. For instance, let $u_1,\cdots ,u_{n_{ba}}\in
    \scr{O}(V)$ be units; then, if $e_i=(0,\dots ,0,1,0,\dots ,0)$,
    the sections $u_1e_1,\dots ,u_{n_{ab}}e_{n_{ba}}$ form a basis.}
  Let $B'=\{e^1,\dots ,e^{n_{ba}}\}$ be the corresponding dual basis
  for $\Gamma_{ba}^*$. Then, in terms of this basis we have $\pi_b^a
  (\sigma )=\sum_ie_i\sigma e^i$. Let $D=\{f_1,\dots ,f_{n_{ba}}\}$ be
  another basis over $V$ with dual basis $D'$. We then have
$$f_i=\sum_j\lambda_{ij}e_j \quad \text{and} \quad f^i=\sum_j\mu^{ij}e^j.$$
Replacing these linear combinations in the equality
$\delta_{ij}=f^i(f_j)$ we obtain
$$\delta_{ij}=\sum_k\mu^{ik}\lambda_{jk}.$$
If $A:=(\lambda_{ij})$ and $B:=(\mu^{ij})$ then the previous equality
implies that $AB^t=I$ or, equivalently, $A^tB=I$, which in terms of
the coefficients is expressed by
$\delta_{ij}=\sum_k\lambda_{ki}\mu^{kj}$. We now compute
$$
\begin{aligned}
  \sum_if_i\sigma f^i &= \sum_i\Bigl (\sum_j\lambda_{ij}e_j\Bigr )\sigma \Bigl (\sum_k\mu^{ik}e^k\Bigr ) 
  = \sum_{j,k}\Bigl (\sum_i\lambda_{ij}\mu^{ik}\Bigr )e_j\sigma e^k \\
  & =\sum_{j,k}\delta_{jk}e_j\sigma e^k 
  =\sum_je_j\sigma e^j, 
\end{aligned}
$$
as desired.
\end{proof}

Then, when defining $\pi_b^a$ locally on each $V$, we have that, by
the previous computation, these expressions coincide over non-empty
overlaps, and thus can be glued together to obtain a morphism over
$U_\alpha \in \mathfrak{U}$
$$\pi_b^a :\Gamma_{aa}\longrightarrow \Gamma_{bb}.$$
This final layer of structure is included in the following

\begin{defi}\label{cardy_fib}
  A Calabi-Yau fibration $\scr{B}$ is called a \emph{Cardy fibration}
  if and only if the following condition, called the \emph{Cardy
    condition}, holds for each open subset $U_\alpha \in
  \mathfrak{U}$: For $a,b\in \scr{B}(U_\alpha )$,
$$\pi^a_b=\iota_b\iota^a.$$
In other words, the following triangle
$$
\xymatrix{
  \Gamma_{aa} \ar[dr]_{\iota^a}\ar[rr]^{\pi^a_b} & & \Gamma_{bb} \\
  & \scr{T}_U \ar[ur]_{\iota_b} & }
$$
should commute.
\end{defi}

We shall deal with Cardy fibrations all along.

\begin{defi}
  A Cardy fibration $\scr{B}$ is said to be \emph{trivializable} if
  conditions (3), (4)a-c in definition \ref{cy_fibration}
  and the Cardy condition hold also for any open subset of each
  $U_\alpha \in \mathfrak{U}$.
\end{defi}

\medskip

\section{Algebraic Properties of Maximal Cardy Fibrations}
\label{sec:algebr-prop-maxim}

This section will be devoted to describing in detail the stack of
boundary conditions $\mathscr{B}$. We will first deal with morphisms
and later with the whole category.

\subsection{Local characterization of categories of branes}
\label{sec:local-char-categ}

The main idea now is to pick a point $x\in M$ and prove that all the
fibres over $x$ of the sheaves involved in this discussions define a
brane category in the sense of Moore and Segal. This approach will let
us generalize all the results of~\cite{moore_segal1} to Cardy fibrations.

Let us fix a point $x\in M$, and assume that $x\in U_\alpha$, where
$U_\alpha$ is semisimple. Given arbitrary labels $a,b\in
\mathscr{B}(U_\alpha )$, let us denote by $E_{ab}$ the fibre over $x$
for the sheaf $\Gamma_{ab}$. We need to show that the
vector spaces $T_xM$ and $E_{ab}$, together with the appropriate
morphisms, form a CY category in the sense of Moore and Segal.

Let us denote by $p_{ab}$ (or just $p$ if the labels are clear) the
sequence of proyections
\begin{equation}\label{proy}
  \Gamma_{ab}(U_\alpha )\longrightarrow \Gamma_{ab,x}\longrightarrow E_{ab},
\end{equation}
where $\Gamma_{ab,x}$ is the stalk over $x$ of the sheaf
$\Gamma_{ab}$. Let $1_a$ be the unit in $\Gamma_{aa}(U_\alpha )$; let
us identify a label $a\in \scr{B}(U_\alpha )$ with $1_a$, and denote
$p_{aa}(1_a)$ by $\overline{a}$. We now define the category of
boundary conditions $\overline{\scr{B}}_x$; its objects are given by
\begin{displaymath}
  \operatorname{Obj}\overline{\mathscr{B}}_x=
  \{\overline{a}=p_{aa}(1_a) \; | \; a\in \mathscr{B}(U_\alpha )\}.
\end{displaymath}
If $\overline{a},\overline{b}\in \overline{\mathscr{B}}_x$, consider
the corresponding units $1_a\in \Gamma_{aa}(U_\alpha )$ and $1_b\in
\Gamma_{bb}(U_\alpha )$. Then we define 
\begin{displaymath}
  \operatorname{Hom}_{\overline{\mathscr{B}}_x}(\overline{a},\overline{b})=E_{ab}.
\end{displaymath}
With this definition,
$\operatorname{Hom}_{\overline{\mathscr{B}}_x}(\overline{a},\overline{b})$
is a $\comp$-vector space, with dimension equal to the rank of
$\Gamma_{ab}$. We shall denote this vector space by
$O_{\overline{a}\overline{b}}$.

We also have the linear forms $\theta :\scr{T}_M\to \scr{O}$ and
$\theta_a:\Gamma_{aa}\to \scr{O}$ which induce linear maps on the
fibres
\begin{displaymath}
  \begin{aligned}
  \overline{\theta}_x &: T_xM\longrightarrow \comp \\
  \theta_{\overline{a}} &: O_{\overline{a}\overline{a}}\longrightarrow \comp \\
\end{aligned}
\end{displaymath}
which provide $T_xM$ and $O_{\overline{a}\overline{a}}$ with a
Frobenius $\comp$-algebra structure.

In the same fashion, the transition morphisms $\iota_a$ and $\iota^a$
induce maps
\begin{displaymath}
  T_xM\stackrel{\iota_{\overline{a}}}{\longleftarrow}
  O_{\overline{a}\overline{a}}\stackrel{\iota^{\overline{a}}}{\longrightarrow}T_xM.
\end{displaymath}

\begin{lemma}
  Let $x_0,x_1\in U_\alpha$. If $U_{\alpha}$ is connected, then the categories
  $\overline{\scr{B}}_{x_0}$ and $\overline{\scr{B}}_{x_1}$ are
  isomorphic.
\end{lemma}
\begin{proof}
  Let us consider two labels $a,b\in \scr{B}(U_\alpha )$; to
  distinguish between the two fibres, let $F_x(\scr{M})$ be the fibre
  over $x$ of the locally free module $\scr{M}$; likewise, let us
  denote by $p_{aa}^0$ (for $x_0$) or $p_{aa}^1$ (for $x_1$) the
  projection \eqref{proy}. By connectivity assumptions, the ranks of
  $\Gamma_{aa}$ and $\Gamma_{ab}$ are constant and we can therefore
  fix isomorphisms
  \begin{displaymath}
    \phi_{aa}:F_{x_0}(\Gamma_{aa})\cong F_{x_1}(\Gamma_{aa})
    \quad \text{and} \quad \phi_{ab} :F_{x_0}(\Gamma_{ab})\cong F_{x_1}(\Gamma_{ab})
  \end{displaymath}
  such that the diagrams
  \begin{displaymath}
  \xymatrix{
  & F_{x_0}(\Gamma_{aa}) \ar[dd]^{\phi_{aa}} \\
  \Gamma_{aa}(U_\alpha )\ar[ur] \ar[dr] & \\
  & F_{x_1}(\Gamma_{aa})} \xymatrix{
  & F_{x_0}(\Gamma_{ab}) \ar[dd]^{\phi_{ab}} \\
  \Gamma_{ab}(U_\alpha )\ar[ur] \ar[dr] & \\
  & F_{x_1}(\Gamma_{ab})}
  \end{displaymath}
  commute, where the unlabelled arrows are canonical projections. In
  particular, this commutativity implies that, for example,
  $p_{aa}^0(1_a)\in F_{x_0}(\Gamma_{aa})$ is mapped onto
  $p_{aa}^1(1_a)$.

We now define a functor $F:\overline{\scr{B}}_{x_0}\to
\overline{\scr{B}}_{x_1}$; on objects, if
$\overline{a}_0:=p_{aa}^0(1_a)$, then
$$F(\overline{a}_0)=\phi_{aa}(\overline{a}_0).$$
Let now $\sigma :\overline{a}_0\to \overline{b}_0$ be an arrow in
$\overline{\scr{B}}_{x_0}$. That is, $\sigma$ is an element of
$F_{x_0}(\Gamma_{ab})$. Then we define
$$F (\sigma )=\phi_{ab}(\sigma ).$$
The inverse of this functor is constructed in the same way, by
considering $\phi_{aa}^{-1}$ and $\phi_{ab}^{-1}$.
\end{proof}

\begin{theorem}\label{ms_over_point}
  The category $\overline{\mathscr{B}}_x$, together with the Frobenius
  algebra $T_xM$ and the structure maps $\overline{\theta}_x$,
  $\theta_{\overline{a}}$, $\iota_{\overline{a}}$ and
  $\iota^{\overline{a}}$ ($\overline{a}\in \overline{\scr{B}}_x$)
  defines a brane category in the sense of Moore and Segal.
\end{theorem}

From theorem \ref{ms_over_point} we can deduce the following
% Hasta aqui 2014-07-28 09:49
\begin{theorem}\label{theorem2}
  Let $a\in \scr{B}(U_\alpha )$ with $U_{\alpha}$ connected. Then, the
  sheaf $\Gamma_{aa}$ is locally isomorphic to a sum
  $\bigoplus_i\operatorname{M}_{d(a,i)}(\mathscr{O}_U)$ of matrix
  algebras
\end{theorem}
\begin{proof}
  Fix $x_0\in U_\alpha$ and let $\{e_1,\dots ,e_n\}$ be a frame of
  orthogonal, idempotent sections in $\scr{T}(U_\alpha )$. Then, for
  the category $\overline{\scr{B}}_{x_0}$, we have Moore and Segal's
  Theorem 2 (\ref{theorem_2}) at our disposal. We have that
  $O_{\overline{a}\overline{a}}=\bigoplus_i\iota_{\overline{a}}
  (e_i(x_0))O_{\overline{a}\overline{a}}$;
  by \ref{theorem_2},
  \begin{equation}\label{theorem2_ms_point}
    O_{\overline{a}\overline{a}}=\operatorname{Hom}_{\overline{\scr{B}}_{x_0}}
    (\overline{a},\overline{a})\cong 
    \bigoplus_{i=1}^n\text{M}_{d(x_0,\overline{a},i)}(\comp );
  \end{equation}
  moreover, the matrix algebra $\opnm{M}_{d(x_0,\overline{a},i)}(\comp
  )$ corresponds to the summand
  $\iota_{\overline{a}}(e_i(x_0))O_{\overline{a}\overline{a}}$. On the
  other hand, we have that, locally around $x_0$, the sheaf
  $\Gamma_{aa}$ is isomorphic to $\scr{O}^{n_a}_{U_\alpha}$ for some
  integer $n_a$. We should link this isomorphism with the pointwise
  decomposition given in equation \eqref{theorem2_ms_point}.

  It is sufficient to work with only one idempotent; we thus consider
  the algebra $\iota_a(e_i)\Gamma_{aa}$, which is a locally free
  module, being the image of the idempotent map $L_i:\Gamma_{aa}\to
  \Gamma_{aa}$ given by $L_i(\sigma )=\iota_a(e_i)\sigma$. Assume that
  $x_0\in V$, where $V$ is a neighborhood such that
  $\iota_a(e_i)\Gamma_{aa}|_V\cong \scr{O}^{n_a(i)}_V$. The fibre over
  $x_0$ of $\iota_a(e_i)\Gamma_{aa}$ is precisely
  $\iota_{\overline{a}}(e_i(x_0))O_{\overline{a}\overline{a}}$, which
  is isomorphic to $\opnm{M}_{d(x_0,\overline{a},i)}(\comp )$. If
  $x_1\in V$ is another point, then
  $\iota_{\overline{a}}(e_i(x_1))O_{\overline{a}\overline{a}}$ is
  isomorphic to $\opnm{M}_{d(x_1,\overline{a},i)}(\comp )$. But the
  local triviality of $\iota_a(e_i)\Gamma_{aa}$ implies that
  \begin{displaymath}
    d(x_1,\overline{a},i)=n_a(i)=d(x_0,\overline{a},i).
  \end{displaymath}
  Hence, by remark~\ref{obs:az}, the decomposition
  \eqref{theorem2_ms_point} extends to a neighborhood of $x_0$, as we
  wanted to prove.
\end{proof}

\begin{obs}\label{remark_summands}
  From the previous result we can also deduce that the matrix algebra
  $M_{d(a,i)}(\scr{O}_V)$ corresponds (locally) to the subalgebra
  $\iota_a(e_i)\Gamma_{aa}$.
\end{obs}

For $a,b\in \scr{B}(U_\alpha )$, and again by the CY structure of
$\overline{\scr{B}}_x$, we have an isomorphism
\begin{displaymath}
  O_{\overline{a}\overline{b}}=\operatorname{Hom}_{\overline{\scr{B}}_x}
  (\overline{a},\overline{b})\cong \bigoplus_{i=1}^n\operatorname{Hom}_\comp 
  \left (\comp^{d(\overline{a},i)},\comp^{d(\overline{b},i)}\right ),
\end{displaymath}
and thus the following result, which is proved following the same
procedure of the previous theorem (note that in this case we have the
idempotent morphism $L_i:\Gamma_{ab}\to \Gamma_{ab}$, $L_i(\sigma
)=\iota_b(e_i)\sigma$ which, by the centrality condition
\eqref{centrality}, coincides with the morphism $\Gamma_{ab}\to
\Gamma_{ab}$ given by $\sigma \mapsto \sigma \iota_a(e_i)$).

\begin{theorem}\label{theorem2bis}
  In the situation of theorem \ref{theorem2}, for $a,b\in
  \mathscr{B}(U_\alpha )$ we have a local isomorphism between
  $\Gamma_{ab}$ and
  $\bigoplus_{i=1}^n\operatorname{Hom}_{\mathscr{O}_U}\left
    (\mathscr{O}_U^{d(a,i)},\mathscr{O}_U^{d(b,i)}\right )$.
\end{theorem}

\begin{obs}\label{remark_summands_2}
  Observe that the dimensions $d(a,i)$ in theorem \ref{theorem2bis}
  are the same as the ones in \ref{theorem2}; this is deduced form the
  proof of Moore and Segal's theorem 2 in \cite{moore_segal1}. And
  also in this case, the summand
  $\operatorname{Hom}_{\mathscr{O}_V}\left
    (\mathscr{O}_V^{d(a,i)},\mathscr{O}_V^{d(b,i)}\right )$
  corresponds to the submodule
  $\iota_b(e_i)\Gamma_{ab}|_V=\Gamma_{ab}|_V\iota_a(e_i)$.
\end{obs}

From these last results, and following the same procedures done in
section \ref{subsec_boundary_semisimple}, we can derive local
expressions for the morphisms $\theta_a$, $\iota^a$ and $\pi^a_b$. Let
$a,b\in \scr{B}(U_\alpha)$ and let $x\in U_\alpha$. Assume that $U\ni
x$ is a neighborhood such that $\Gamma_{aa}|_U$ is isomorphic to a sum
$\bigoplus_i\opnm{M}_{d(a,i)}(\scr{O}_U)$ (in that case an element
$\sigma \in \Gamma_{aa}|_U$ can be represented as a tuple
$(\sigma_i)$, where $\sigma_i\in \opnm{M}_{d(a,i)}(\scr{O}_U)$). If
$\{e_1,\dots ,e_n\}$ is a frame of orthogonal, idempotent sections for
$\scr{T}_M$ over $U_\alpha$, then we have the following expressions
for $\theta_a$, $\iota^a$ and $\pi^a_b$ over $U$:
\begin{equation}\label{local_expressions}
  \begin{aligned}
    \theta_a (\sigma ) &= \sum_i\sqrt{\theta (e_i)} \opnm{tr}(\sigma_i), \\
    \iota^a(\sigma )   &= \sum_i\frac{\opnm{tr}(\sigma_i)}{\sqrt{\theta (e_i)}}e_i, \\
    \pi_b^a(\sigma )   &= \sum_i\frac{\opnm{tr}(\sigma_i)}{\sqrt{\theta (e_i)}}\iota_b(e_i). \\
  \end{aligned}
\end{equation}

\subsection{Enlarging the categories of branes}
\label{sec:enlarg-categ-bran}

Let $\mr{B}$ be a Cardy fibration over a manifold $M$. 

\subsubsection{Additive structure}
\label{sec:additive-structure}

We shall show that $\mr{B}$ can be embedded in a canonical way into a
fibration of additive Cardy categories. 
Let $U\subset M$ be any open subset and $a,b,c\in \scr{B}(U)$; based
on properties of modules, we shall define a new label $a\oplus b$; we
put
\begin{displaymath}
  \begin{aligned}
    \Gamma_{(a\oplus b)c}:=\Gamma_{ac}\oplus \Gamma_{bc}, \\
    \Gamma_{c(a\oplus b)}:=\Gamma_{ca}\oplus \Gamma_{cb}.\\
  \end{aligned}
\end{displaymath}
A morphism $a\oplus b\to c$ shall be represented as a row matrix $
\left ( \begin{smallmatrix} \sigma & \tau \\ \end{smallmatrix} \right
) $ , where $\sigma :a\to c$, $\tau :b\to c$. Likewise, an arrow $c\to
a\oplus b$ is a column matrix 
$
\left ( \begin{smallmatrix} \sigma \\
    \tau \\ \end{smallmatrix} \right )$, for $\sigma :c\to a$, $\tau
:c\to b.  $ 
Thus, a map $a_1\oplus a_2\to b_1\oplus b_2$ can be
represented as a matrix $\left (\begin{smallmatrix} \sigma_{11} &
    \sigma_{21} \\ \sigma_{12} & \sigma_{22} \\ \end{smallmatrix}
\right )$, where $\sigma_{ij}:a_i\to b_j$. Composition of maps is then
given by multiplying matrices. As a consequence, we obtain thus a
structure of additive category for each $\scr{B}(U)$.
For a new object $a\oplus b$ we define $\theta_{a\oplus
  b}:\Gamma_{(a\oplus b)(a\oplus b)}\to \scr{O}_U$ by
\begin{equation}\label{linear_form_additive}
  \theta_{a\oplus b}\left (
  \begin{smallmatrix} \sigma_{11} & \sigma_{21} 
    \\ \sigma_{12} & \sigma_{22} 
  \end{smallmatrix} 
 \right )=\theta_{a}(\sigma_{11})+\theta_b(\sigma_{22}).
\end{equation}

Regarding nondegeneracy of the linear forms we have the following

\begin{proposition}\label{nondeg_additive}
The diagram
\begin{displaymath}
  \xymatrix{
    \Gamma_{(a\oplus b)c}\otimes \Gamma_{c(a\oplus b)}\ar[d]\ar[r] 
    & \Gamma_{(a\oplus b)(a\oplus b)}
    \ar[r]^-{\theta_{a\oplus b}} & \scr{O}_U\ar@{=}[d] \\
    \Gamma_{c(a\oplus b)}\otimes \Gamma_{(a\oplus b)c}\ar[r] &
  \Gamma_{cc} \ar[r]^{\theta_c} & \scr{O}_U
}
\end{displaymath}
is commutative, and the top and botton composite bilinear maps are
non-degenerate parings (the vertical arrow on the left is the twisting
map).
\end{proposition}

We now define the transition morphisms $\iota_{(a\oplus
  b)}:\scr{T}_{U_\alpha }\to \Gamma_{(a\oplus b)(a\oplus b)}$ and
$\iota^{(a\oplus b)}:\Gamma_{(a\oplus b)(a\oplus b)}\to
\scr{T}_{U_\alpha }$ by the equations
\begin{equation}\label{transition_additive}
\begin{aligned}
  \iota_{(a\oplus b)}(X) &= \left (
    \begin{smallmatrix} \iota_a(X) & 0 \\ 
      0 & \iota_b(X) \\ 
    \end{smallmatrix} \right ), \\
  \iota^{(a\oplus b)}\left (
    \begin{smallmatrix} \sigma_{11} & \sigma_{21} \\ 
      \sigma_{12} & \sigma_{22} \\ 
    \end{smallmatrix} \right ) &=\iota^a(\sigma_{11})+\iota^b(\sigma_{22}).\\
\end{aligned}
\end{equation}
In particular, note that both $\iota_{a\oplus b}$ and $\iota^{a\oplus
  b}$ are $\scr{O}_{U_\alpha }$-linear, and $\iota_{a\oplus b}$ is an
algebra homomorphism which preserves the unit.

The following result shall be useful to prove the Cardy condition.

\begin{lemma}\label{cardy_additive}
For the maps $\pi^{a\oplus b}_c$ and $\pi^a_{b\oplus c}$ the following equalities hold
\begin{displaymath}
  \begin{aligned}
    \pi^{a\oplus b}_c &= \pi^a_c+\pi^b_c \\
    \pi^a_{b\oplus c} &= \left (
      \begin{smallmatrix} \pi^a_b & 0 \\ 0 & \pi^a_c \\ 
      \end{smallmatrix} \right ). 
\end{aligned}
\end{displaymath}
\end{lemma}

\begin{theorem}\label{maximal_additive}
  Given $a,b \in \scr{B}(U_\alpha )$, the maps $\theta_{a\oplus b}$,
  $\iota_{(a\oplus b)}$ and $\iota^{(a\oplus b)}$ verify the
  centrality, adjoint and Cardy conditions.
\end{theorem}

\begin{proof}
  For the centrality condition, take $\sigma :a\oplus b\to c$, which
  can be represented by a matrix $\left (\begin{smallmatrix}
      \sigma_{11} & \sigma_{21} \\ \end{smallmatrix}\right )$. Then
  \begin{displaymath}
    \begin{aligned}
      \sigma \iota_{a\oplus b}(X) &= \left (
        \begin{smallmatrix} 
          \sigma_{11} & \sigma_{21} \\ 
        \end{smallmatrix}\right ) 
      \left (
        \begin{smallmatrix} 
          \iota_a(X) & 0 \\ 0 & \iota_b(X) \\ 
        \end{smallmatrix}\right )\\
   &= \left (
     \begin{smallmatrix} 
       \sigma_{11}\iota_a(X) & \sigma_{21}\iota_b(X) 
     \end{smallmatrix}\right ).
 \end{aligned}
\end{displaymath}
The equality $\sigma \iota_{a\oplus b}(X)=\iota_c(X)\sigma$ now
follows from the centrality condition for the morphisms
$\iota_a,\iota_c$ and $\iota_b,\iota_c$.

We now verify the adjoint relation $ \theta_{a\oplus b}\left (\sigma
  \iota_{a\oplus b}(X)\right ) =\theta \left (\iota^{a\oplus b}(\sigma
  )X\right );$ so let $\sigma :a\oplus b\to a\oplus b$ be given by
$(\sigma_{ij})^t$. Then the adjoint relation between $\iota_a,\iota^a$
and the one between $\iota_b\iota^b$ let us write
\begin{displaymath}
  \begin{aligned}
    \theta_{a\oplus b}\left (\sigma \iota_{a\oplus b}(X)\right ) 
    &= \theta_{a\oplus b}\left (
      \begin{smallmatrix} 
        \sigma_{11}\iota_a(X) & \sigma_{21}\iota_b(X) \\ 
        \sigma_{12}\iota_a(X) & \sigma_{22}\iota_b(X) \\ 
      \end{smallmatrix} \right ) \\  
    &= \theta_a\left (\sigma_{11}\iota_a(X)\right )+\theta_b\left (\sigma_{22}\iota_b(X)\right ) \\
    &= \theta \left (\iota^a(\sigma_{11})X\right )+\theta \left (\iota^b(\sigma_{22})X\right ) \\
    &= \theta \left (\left (\iota^a(\sigma_{11})+\iota^b(\sigma_{22})\right )X\right ) \\
    &= \theta \left (\iota^{a\oplus b}(\sigma )X\right ), \\
  \end{aligned}
\end{displaymath}
as desired.

For the Cardy condition, we now check that $\pi^{a\oplus b}_{c\oplus d}=\iota_{c\oplus d}\iota^{a\oplus b}$. The right hand side is
\begin{displaymath}
  \begin{aligned}
  \iota_{c\oplus d}\iota^{a\oplus b}\left (
    \begin{smallmatrix} 
      \sigma_{11} & \sigma_{21} \\ \sigma_{12} & \sigma_{22} \\ 
    \end{smallmatrix} \right ) &= 
  \iota_{c\oplus d}\left (\iota^a(\sigma_{11}) + \iota^b(\sigma_{22})\right ) \\
  &= \left (\begin{smallmatrix} \iota_c\left (\iota^a(\sigma_{11})+\iota^b(\sigma_{22})\right ) & 0 \\ 
      0 & \iota_d\left (\iota^a(\sigma_{11})+\iota^b(\sigma_{22})\right ) \\ \end{smallmatrix} \right ) \\
  &= \left (\begin{smallmatrix} \pi^a_c(\sigma_{11})+\pi^b_c(\sigma_{22}) & 0 \\ 
      0 & \pi^a_d(\sigma_{11})+\pi^b_d(\sigma_{22}) \\ \end{smallmatrix} \right ), \\
\end{aligned}
\end{displaymath}
where in the last equality we used the Cardy condition. The rest now follows from lemma~\ref{cardy_additive}.
\end{proof}

\begin{cor}
Any maximal Cardy fibration is additive.
\end{cor}

\subsubsection{The Action of the Category of Locally Free sheaves}

In this section we shall prove that another enlargement of the
category $\scr{B}$ can be made, by considering a label of the form
$\scr{M}\otimes a$, where $\scr{M}$ is a locally free
$\scr{O}_U$-module and $a\in \scr{B}(U)$. A consequence of this
construction is that every maximal fibration enjoys, besides an
additive structure, an action of the (fibred) category of locally free
sheaves, which is compatible with the additive structure.

So let the locally free $\scr{O}_U$-module $\mathscr{M}$ be given, as
well as a brane $a\in \scr{B}(U)$ over $U$. The new product brane
$\scr{M}\otimes a$ is defined by
\begin{equation}\label{product_brane}
\begin{aligned}
  \Gamma_{(\scr{M}\otimes a)b} &= \scr{M}^*\otimes \Gamma_{ab}, \\
\Gamma_{b(\scr{M}\otimes a)} &= \scr{M} \otimes \Gamma_{ba}, \\
\end{aligned}
\end{equation}
where the tensor product is taken over $\scr{O}_U$. In particular, we
also have that
\begin{displaymath}
  \Gamma_{(\scr{M}\otimes a)(\scr{N}\otimes b)}=\underline{\opnm{Hom}}(\scr{M},\scr{N})\otimes \Gamma_{ab},
\end{displaymath}
by the canonical identification between $\scr{M}^*\otimes \scr{N}$ and
$\underline{\opnm{Hom}}(\scr{M},\scr{N})$ (so an object of the form
$\varphi \otimes x$ shall be regarded as a homomorphism $\scr{M}\to
\scr{N}$). Note that this definition let us also define a restriction
$(\scr{M}\otimes a)|_V:=\scr{M}|_V\otimes a|_V$. Moreover, if we work
on a semisimple subset $U_\alpha \in \mathfrak{U}$, then
$\Gamma_{(\scr{M}\otimes a)b}$ and $\Gamma_{b(\scr{M}\otimes a)}$ are
locally free.

The composition pairing
\begin{equation}\label{pairing_prod}
  \Gamma_{(\scr{M}\otimes a)(\scr{N}\otimes b)}\otimes 
  \Gamma_{(\scr{N}\otimes b)(\scr{P}\otimes c)}\longrightarrow 
  \Gamma_{(\scr{M}\otimes a)(\scr{P}\otimes c)}
\end{equation}
can be also written as
\begin{displaymath}
  \scr{M}^*\otimes \scr{N} \otimes \scr{N}^* \otimes \scr{P}\otimes 
  \Gamma_{ab}\otimes \Gamma_{bc}\longrightarrow \scr{M}^*\otimes \scr{P}\otimes \Gamma_{ac};
\end{displaymath}
hence, the map \eqref{pairing_prod} is built from two composition
pairings, the one corresponding to composition of module
homomorphisms, namely $\scr{M}^*\otimes \scr{N}\otimes
\scr{N}^*\otimes \scr{P}\to \scr{M}^*\otimes \scr{P}$, and the one
corresponding to composition of maps of branes, $\Gamma_{ab}\otimes
\Gamma_{bc}\to \Gamma_{ac}$.
\begin{lemma}
  We have a duality isomorphism $\Gamma_{(\scr{M}\otimes a)b}\cong
  \Gamma_{b\otimes (\scr{M}\otimes a)}^*$.
\end{lemma}
\begin{proposition}
The correspondence $(\scr{M},a)\mapsto \scr{M}\otimes a$ defines an action
\begin{displaymath}
  \tsf{LF}_{\scr{O}_U}\times \scr{B}(U)\longrightarrow \scr{B}(U)
\end{displaymath}
which is compatible with the additive structure.
\end{proposition}
\begin{theorem}\label{action_cy}
  With the previous definitions, the action
  $\tsf{LF}_{\scr{O}_U}\times_U\scr{B}|_U\to \scr{B}|_U$ is compatible
  with all the structures in a Cardy fibration.
\end{theorem}
We thus obtain the following

\begin{cor}
Any maximal CY category $\scr{B}$ over $M$ comes equipped with a linear action $\tsf{LF}_{\scr{O}_M}\times \scr{B}\to \scr{B}$.
\end{cor}  

% Hasta aqui 2014-08-11
\subsubsection{Pseudo-Abelian Structure}

We shall now show that besides the additive structure and the action
of the category of locally free sheaves, any maximal Cardy fibration
should be pseudo-abelian.  That is to
say, given $a\in \scr{B}(U)$ and an arrow $\sigma_0 :a\to a$ such that
$\sigma_0^2=\sigma_0$, we shall assume that there exists branes
$K_0:=\opnm{Ker}\sigma_0$ and $I_0:=\opnm{Im}\sigma_0$ (which can also
be taken as $\opnm{Ker}(1_a-\sigma_0 )$) such that
\begin{itemize}
\item The brane $a$ decomposes as a sum $a\cong K_0 \oplus I_0$ and
\item using matrix notation, the map $\sigma_0$ is given by 
  $\left (\begin{smallmatrix} 0 & 0 \\ 0 & 1_a \\ \end{smallmatrix} \right )$.
\end{itemize}
As was done for the additive structure and the action of the category
of locally free modules, the enlargement of the category of branes by
adding kernels should be done by defining all the structure maps for
this new object $K_0$, namely $\theta_{K_0}$, $\iota_{K_0}$,
$\iota^{K_0}$, along with the verification of their properties. In
particular, it should be noted that this definitions should agree with
the additive structure.

First note that an arrow $K_0\to K_0$ is a composite of the form
\begin{displaymath}
  K_0\stackrel{i_1}{\longrightarrow}K_0\oplus 
  I_0\stackrel{\sigma}{\longrightarrow}K_0
  \oplus I_0\stackrel{\pr_1}{\longrightarrow}K_0
\end{displaymath}
for some arrow $\sigma :a\to a$, and hence $\Gamma_{K_0K_0}\subset
\Gamma_{aa}$ is a submodule. In fact, we have that
\begin{displaymath}
  \Gamma_{aa}=\Gamma_{K_0K_0}\oplus \Gamma_{K_0I_0}\oplus \Gamma_{I_0K_0}\oplus \Gamma_{I_0I_0}.
\end{displaymath}
For $a\in \scr{B}(U_\alpha )$, consider the homomorphism $\rho :\Gamma_{aa}\to \Gamma_{aa}$ given by
\begin{displaymath}
  \rho \left (
    \begin{smallmatrix} 
      \sigma_{11} & \sigma_{21} \\ 
      \sigma_{12} & \sigma_{22} 
    \end{smallmatrix} 
  \right )=\left (
    \begin{smallmatrix} 
      0 & \sigma_{21} \\ 
      \sigma_{12} & \sigma_{22} 
    \end{smallmatrix} 
  \right ).
\end{displaymath}
Then $\rho$ is clearly a projection with kernel $\Gamma_{K_0K_0}$
which is then locally-free. A similar argument can be used to prove
that for any label $b\in \scr{B}(U_\alpha )$, $\Gamma_{K_0b}$ is also
locall free; consider $\Gamma_{ab}=\Gamma_{K_0b}\oplus \Gamma_{I_0b}$
and the map $\eta :\Gamma_{ab}\to \Gamma_{ab}$ which projects to
$\Gamma_{I_0b}$. Proposition~\ref{non_deg_pseudoab} shows that also
$\Gamma_{bK_0}\cong \Gamma_{K_0b}^*$ is locally free.

We now turn to the structure maps. If $a\cong K_0\oplus I_0$, the fact
that
\begin{displaymath}
  \theta_a\left (
    \begin{smallmatrix} 
      0 & \sigma_{21} \\ 
      0 & 0 
    \end{smallmatrix} 
  \right )=\theta_a\left (
    \begin{smallmatrix} 
      0 & 0 \\ 
      \sigma_{12} & 0 
    \end{smallmatrix} \right )=0
\end{displaymath}
suggests the definition of the linear form $\theta_{K_0}:\Gamma_{K_0K_0}\to \scr{O}_U$ by
\begin{displaymath}
  \theta_{K_0}(\sigma )=\theta_a\left (
    \begin{smallmatrix} 
      \sigma & 0 \\ 
      0      & 0 
    \end{smallmatrix} 
  \right ).
\end{displaymath}

\begin{proposition}\label{non_deg_pseudoab}
The diagram
\begin{displaymath}
  \xymatrix{
    \Gamma_{K_0b}\otimes \Gamma_{bK_0}\ar[d]\ar[r] & 
    \Gamma_{K_0K_0}\ar[r]^-{\theta_{K_0}} & 
    \scr{O}_U\ar@{=}[d] \\
    \Gamma_{bK_0}\otimes \Gamma_{K_0b}\ar[r] & 
    \Gamma_{bb} \ar[r]^{\theta_b} & \scr{O}_U}
\end{displaymath}
is commutative, and the top and botton composite bilinear maps are
non-degenerate parings (the vertical arrow on the left is the twisting
map).
\end{proposition}

\begin{lemma}
We have $\varphi_{12}=\varphi_{21}=0$.
\end{lemma}

\begin{theorem}
  The maps $\theta_{K_0}$, $\iota_{K_0}$ and $\iota^{K_0}$ satisfy the
  centrality, adjoint and Cardy conditions.
\end{theorem}

Hence, we obtain the following

\begin{cor}
Any maximal CY category $\scr{B}$ over $M$ is pseudo-abelian.
\end{cor}

\subsubsection{Local Structure of a maximal Cardy Fibration}
\label{ss_local_structure}

There is a further assumption to be made about maximal categories in order to obtain a full description.

\begin{defi}
  Let $U\subset M$ be a semisimple open subset. We shall say that a
  label $a\in \scr{B}(U)$ is \emph{supported on an index $i_0$} if
  \begin{displaymath}
    \iota_{a}(e_{i_0})=1_a.
  \end{displaymath}
  Equivalently, $\iota_a(e_j)=0$ for each $j\neq i_0$. 
\end{defi}

\begin{lemma}\label{support_ij}
  Let $i\neq j$ be two indices, $1\leqslant i,j\leqslant n$ and let
  $a,b$ be labels over a semisimple open subset of $M$. If $a$ and $b$
  are supported on $i$ and $j$ respectively, then $\Gamma_{ab}=0$.
\end{lemma}
\begin{proof}
  Pick an arrow $\sigma \in \Gamma_{ab}$. Then
  $
    \sigma =\sigma 1_a=\sigma \iota_a(e_i)=\iota_b(e_i)\sigma =0,
  $
  as claimed.
\end{proof}

\begin{lemma}\label{existence_support}
  Let $\scr{B}$ be a maximal category of branes and $U$ a semisimple
  open subset. For each index $i$, $1\leqslant i\leqslant n$, there
  exists a label $\xi_i$ supported on $i$.
\end{lemma}
\begin{proof}
  Assume that this statement is false. We shall see that the
  maximality of $\scr{B}$ will not allow this to happen.

  So we first assume that $\iota_a(e_j)=0$ for each index $j$ and each
  $a\in \scr{B}(U)$. We define a new category $\scr{C}$: the objects
  of $\scr{C}(U)$ are objects of $\scr{B}(U)$ plus one label, which we
  denote by $\xi_i$. We also define
\begin{itemize}
\item $\Gamma_{\xi_i\xi_i}=\scr{O}_U$.
\item $\Gamma_{\xi_ia}=\Gamma_{a\xi_i}=0$; this definition is
  motivated by lemma \ref{support_ij}.
\item $\theta_{\xi_i}:\Gamma_{\xi_i\xi_i}=\scr{O}_U \to \scr{O}_U$ is
  the identity.
\item Let $X=\sum_j\lambda_je_j$ be a local vector field. Then
  $\iota_{\xi_i}:\scr{T}_U\to \Gamma_{\xi_i\xi_i}$ and
  $\iota^{\xi_i}:\Gamma_{\xi_i\xi_i}\to \scr{T}_U$ are given by
  \begin{displaymath}
    \iota_{\xi_i}(X)=\lambda_i \quad \text{and} \quad 
    \iota^{\xi_i}(\lambda )=\lambda e_i.
  \end{displaymath}
\end{itemize}
These definitions make $\scr{C}$ a Cardy fibration, contradicting the
maximality of $\scr{B}$.
\end{proof}

\begin{proposition}\label{invertibles}
  Let $U$ be a semisimple neighborhood. For each index $i=1,\dots ,n$,
  there exists a label $\xi_i\in \scr{B}(U)$ supported on $i$ such
  that $\Gamma_{\xi_i\xi_i}\cong \scr{O}_U$.
\end{proposition}
\begin{proof}
  Let $i$ be an index, $1\leqslant i\leqslant n$. By lemma
  \ref{existence_support}, we can pick a label $a_i$ supported in
  $i$. If $\Gamma_{a_ia_i}\cong \scr{O}_U$, then $\xi_i:=a_i$ is the
  label we are looking for. If not, we have that $\Gamma_{a_ia_i}$ can
  be taken to be a matrix algebra $\opnm{M}_{d_i}(\scr{O}_U)$ (the
  construction of such a label is assured by maximality of the
  category of branes, and can be proved by following exactly the same
  procedure used in the proof of lemma \ref{existence_support}). Let
  then $\sigma \in \Gamma_{a_ia_i}$ be an idempotent matrix, which can
  be regarded as a morphism $\sigma :\scr{O}_U^{d_i}\to
  \scr{O}_U^{d_i}$. Moreover, assume that $\sigma$ is the projection
  \begin{displaymath}
    \sigma (\lambda_1,\dots ,\lambda_n)
    =(\lambda_1,\dots ,\lambda_{i-1},0,\lambda_{i+1},\dots ,\lambda_n).
  \end{displaymath}
  Then, as the category of branes is pseudo-abelian, we have that
  $\opnm{Ker}\sigma \cong \scr{O}_U\in \scr{B}(U)$. As $\scr{O}_U$ is
  indecomposable, we should have
  $\Gamma_{\opnm{Ker}\sigma\opnm{Ker}\sigma}\cong \scr{O}_U$, and
  hence $\xi_i:=\opnm{Ker}\sigma$ is the object we were looking for.
\end{proof}

\begin{lemma}\label{simple_mods}
$\Gamma_{\xi_i\xi_j}=0$ for $i\neq j$.
\end{lemma}
\begin{proof}
This is an immediate consequence of lemma \ref{support_ij}.
\end{proof}

We shall need the following decomposition for $\Gamma_{ab}$.

\begin{proposition}\label{decomposition}
  For labels $a,b\in \mathscr{B}(U)$, with $U$ a semisimple
  neighborhood, we have an isomorphism
  \begin{displaymath}
    \Gamma_{ab}\cong \bigoplus_i\Gamma_{a\xi_i}\otimes\Gamma_{\xi_ib}.
  \end{displaymath}
\end{proposition}
\begin{proof}
  Define the map $\phi
  :\bigoplus_i\Gamma_{a\xi_i}\otimes\Gamma_{\xi_ib}\rightarrow\Gamma_{ab}$
  by
  \begin{equation}\label{iso_decomp}
    \phi (\sigma_1\otimes \tau_1,\dots ,
    \sigma_n\otimes \tau_n)=\sum_i\tau_i\sigma_i.
  \end{equation}
  Using the characterization given in \ref{theorem2bis}, we have a
  local isomorphism
  \begin{displaymath}
    \bigoplus_i\Gamma_{a\xi_i}\otimes\Gamma_{\xi_ib} 
    \cong \bigoplus_i \left (\bigoplus_j
      \underline{\operatorname{Hom}}_{\scr{O}_U}
      \left (\mathscr{O}_U^{d(a,j)},\mathscr{O}_U^{d(\xi_i,j)}\right )\right )
    \otimes \left (\bigoplus_k
      \underline{\operatorname{Hom}}_{\mathscr{O}_U}
      \left (\mathscr{O}_U^{d(\xi_i,k)},\mathscr{O}_U^{d(b,k)}\right )\right ).
  \end{displaymath}
  By \ref{simple_mods}, we have that $d(\xi_i,k)=\delta_{ik}$, and
  thus
  \begin{displaymath}
    \bigoplus_i\Gamma_{a\xi_i}\otimes\Gamma_{\xi_ib} \cong \bigoplus_i 
    \underline{\operatorname{Hom}}_{\scr{O}_U}
    \left(\mathscr{O}_U^{d(a,i)},\mathscr{O}_U\right)
    \otimes \underline{\operatorname{Hom}}_{\scr{O}_U}\left(\mathscr{O}_U,\mathscr{O}_U^{d(b,i)}\right).
  \end{displaymath}
  On the other hand, by \ref{theorem2}, we also have that, locally,
  $\Gamma_{ab}\cong
  \bigoplus_i\underline{\operatorname{Hom}}_{\scr{O}_U}\left
    (\mathscr{O}_U^{d(a,i)},\mathscr{O}_U^{d(b,i)}\right )$. Combining
  these facts with \eqref{iso_decomp} we conclude that the stalk maps
  $\phi_x$ are in fact bijections for each $x\in U$.
\end{proof}

A useful consequence of \ref{decomposition} is the following

\begin{cor}\label{linear_comb}
  For each label $b$ over $U$, we have an isomorphism
  $
    b\cong \bigoplus_i\Gamma_{\xi_ib}\otimes \xi_i.
  $
\end{cor}
\begin{proof}
Take any label $c$. By equations \eqref{product_brane} and duality we have
\begin{displaymath}
  \begin{aligned}
    \underline{\operatorname{Hom}}_U\Bigl
    (\bigoplus_i\Gamma_{\xi_ib}\otimes \xi_i,c\Bigr )
    &\cong \bigoplus_i\Gamma_{b\xi_i}\otimes
    \underline{\operatorname{Hom}}_U(\xi_i,c) \\
    &\cong \bigoplus_i\Gamma_{b\xi_i}\otimes\Gamma_{\xi_ic} \cong \Gamma_{bc}.\\
  \end{aligned}
\end{displaymath}
As $c$ is arbitrary, the result follows.
\end{proof}

Note that the coefficient modules in the previous result are unique,
up to isomorphism: if $b\cong \bigoplus_i\scr{M}_i\otimes \xi_i$, then
\begin{displaymath}
  \Gamma_{\xi_jb}\cong \bigoplus_i\scr{M}_i\otimes
  \Gamma_{\xi_j\xi_i}\cong\scr{M}_j.
\end{displaymath}

The next result addresses some uniqueness issues.

\begin{proposition}\label{uniqueness_labels}
Let $\xi_i\in \scr{B}(U)$ be as in \ref{invertibles}, where $U$ is semisimple.
\begin{enumerate}[(1)]
\item Let $\eta_i$ be a label with the same properties as
  $\xi_i$. Then, there exists an invertible sheaf $\scr{L}$ over $U$
  such that
  $
    \eta_i\cong \scr{L}\otimes \xi_i.
  $
  The converse statement also holds.
\item If $\scr{M}$ is a locally-free module such that $\scr{M}\otimes
  \xi_i\cong \xi_i$, then $\scr{M}\cong \scr{O}_U$.
\end{enumerate}
\end{proposition}
\begin{proof}
  For the first item, by \ref{support_ij} and \ref{linear_comb}, we
  have that
  $ 
  \eta_i \cong \bigoplus_j\Gamma_{\xi_j\eta_i}\otimes \xi_j 
  \cong \Gamma_{\xi_i\eta_i}\otimes \xi_i. 
  $
  Define now $\scr{M}_i=\Gamma_{\xi_i\eta_i}$. Then,
  \begin{displaymath}
      \scr{O}_U \cong \Gamma_{\eta_i\eta_i} \cong 
      \Gamma_{\left (\scr{M}_i\otimes \xi_i \right )\left (\scr{M}_i\otimes \xi_i \right )}
      \cong \scr{M}_i^*\otimes \scr{M}_i\otimes \Gamma_{\xi_i\xi_i} 
      \cong \Gamma_{\xi_i\eta_i}^*\otimes \Gamma_{\xi_i\eta_i}.
  \end{displaymath}
The converse is immediate by properties of the action $\scr{L}\otimes \xi_i$.

% Hasta aqui 19-08-2014

For (2), as $\scr{M}\otimes \xi_i \cong \xi_i$, the modules
$\Gamma_{\xi_i\xi_i}$ and $\Gamma_{\xi_i\left (\scr{M}\otimes
    \xi_i\right )}$ are isomorphic. Hence,
\begin{displaymath}
  \scr{O}_U\cong \Gamma_{\xi_i\left (\scr{M}\otimes \xi_i\right )}
  \cong \scr{M}\otimes \Gamma_{\xi_i\xi_i}\cong \scr{M},
\end{displaymath}
as desired.
\end{proof}

\begin{theorem}\label{equivalences}
There exists an open cover $\mathfrak{U}$ of $M$ and an equivalence of
categories 
\begin{equation}\label{equiv_2vb}
\mathscr{B}(U)\simeq \tsf{LF}^n_{\scr{O}_U}
\end{equation}
for each $U\in \mathfrak{U}$, where $\tsf{LF}^n_{\scr{O}_U}$ denotes
the $n$-fold fibred product of $\tsf{LF}_{\scr{O}_U}$. 
\end{theorem}
\begin{proof}
  Let $\mathfrak{U}=\{U_\alpha \}$ be an open cover of $M$, where each
  $U_\alpha$ is semisimple. Define a functor $F_\alpha:\mathscr{B}(U_\alpha
  )\rightarrow \tsf{LF}^n_{\scr{O}_{U_\alpha}}$ on objects by
  \begin{displaymath}
    F_\alpha (a)=(\Gamma_{\xi_1a},\dots ,\Gamma_{\xi_na}),
  \end{displaymath}
  where the objects $\xi_i$ are the ones of proposition
  \ref{invertibles}, and on arrows by $F_\alpha (\sigma )=\sigma_*$;
  that is, if $\sigma :a\to b$, then $F_\alpha (\sigma )(\tau_1,\dots
  ,\tau_n)=(\sigma \tau_1,\dots ,\sigma \tau_n)$. We now define
  $G_\alpha:\tsf{LF}^n_{\scr{O}_{U_\alpha}}\rightarrow\scr{B}(U_\alpha
  )$ on objects by
  \begin{displaymath}
    G_\alpha (\mathscr{M}_1,\dots ,\mathscr{M}_n)=
    \bigoplus_i \mathscr{M}_i\otimes \xi_i
  \end{displaymath}
  and on arrows by
\begin{displaymath}
  G_\alpha (f_1,\dots ,f_n)=
  (f_1\otimes \opnm{id}_{\xi_1},\dots ,f_n\otimes \opnm{id}_{\xi_n}),
\end{displaymath}
where $f_i:\scr{M}_i\to \scr{N}_i$.

We then have that $F_\alpha G_\alpha (\mathscr{M}_1,\dots
,\mathscr{M}_n)=(\Gamma_{\xi_1\overline{a}},\dots
,\Gamma_{\xi_n\overline{a}})$, where
$\overline{a}:=\bigoplus_j\scr{M}_j\otimes \xi_j$. Now,
\begin{displaymath}
  \begin{aligned}
    \Gamma_{\xi_i\overline{a}} 
    &\cong \bigoplus_j \underline{\operatorname{Hom}}_U
    (\xi_i,\mathscr{M}_j\otimes \xi_j) \\
    &\cong \bigoplus_j \mathscr{M}_j\otimes 
    \underline{\operatorname{Hom}}_U(\xi_i,\xi_j) \\
    &\cong \mathscr{M}_i \\
\end{aligned}
\end{displaymath}
by \eqref{product_brane} and \ref{simple_mods}.

The other way, we have $G_\alpha F_\alpha
(a)=\bigoplus_i\Gamma_{\xi_ia}\otimes \xi_i$, which is isomorphic to
$a$ by \ref{linear_comb}.
\end{proof}

In terms of the spectral cover, over each semisimple $U\subset M$ we
have $\pi^{-1}(U)=\bigsqcup_{i=1}^n\widetilde{U}_i$, where each
$\widetilde{U}_i$ is homeomorpic to $U$ by the projection $\pi :S\to
M$, and thus we can write the $n$-fold product
$\tsf{LF}^n_{\scr{O}_U}$ as the pushout
$(\pi_*\tsf{LF}_{\scr{O}_S})(U)=\tsf{LF}_{\scr{O}_{\pi^{-1}(U)}}$. But
$\scr{O}_{\pi^{-1}(U)}$ is the sheaf $(\pi_*\scr{O}_S)|_U$, which is
in turn isomorphic to the tangent sheaf $\scr{T}_U$ by proposition
\ref{isom_1}. Moreover, if $f:M\to N$ is a continuous map, then, by
definition, the fibred categories $f_*\tsf{LF}_{\scr{O}_M}$ and
$\tsf{LF}_{f_*\scr{O}_M}$ are equal. Thus, combining all these facts
we can deduce that
$$\pi_*\tsf{LF}_{\scr{O}_S}=\tsf{LF}_{\pi_*\scr{O}_S}\simeq \tsf{LF}_{\scr{T}_M}.$$

\begin{cor}
  Given a maximal Cardy fibration $\scr{B}$ over a massive manifold
  $M$, there exists an open cover $\mathfrak{U}$ of $M$ such that the
  category $\scr{B}(U)$ is equivalent to the category
  $\tsf{LF}_{\scr{T}_U}$ of locally free $\scr{T}_U$-modules.
\end{cor}

Before stating the next result, we give a preliminary
definition. Given a vector bundle $E$ we can construct the exterior
powers $\bigwedge^kE$ which for a point $x\in M$ have fibre
$\bigwedge^kE_x$. Given now a bundle map $\phi :E\to F$, we have that
$\phi^{\wedge k}:\bigwedge^kE\to \bigwedge^kF$ is given by
\begin{displaymath}
  \phi^{\wedge k}(e_1\wedge \dots \wedge e_n)=
  \phi (e_1)\wedge \dots \wedge \phi (e_n).
\end{displaymath}
After this brief comment about exterior powers, we can now give the
definition we need (see \cite{audin:frob} and references cited
therein). A \emph{Higgs pair} for a manifold $M$ is a pair $(E,\phi
)$, where $E$ is a vector bundle and $\phi :TM\to \opnm{End}(E)$ is a
morphism such that $\phi \wedge \phi =0$. This last condition is
expressing that for each $x\in M$, the endomorphisms $\phi_x(v)\in
\opnm{End}(E_x)$ (for $v\in T_xM$) commute.

\begin{cor}
  Given $a\in \scr{B}(U_\alpha )$, the transition homomorphism
  $\iota_a$ consists of $n$ Higgs pairs for $U_\alpha$.
\end{cor}

The meaning of ``consists of $n$ Higgs pairs'' is explained in the
following proof.

\begin{proof}
  From theorem \ref{equivalences}, we have an equivalence $F_\alpha
  :\scr{B}(U_\alpha )\to \tsf{LF}_{\scr{O}_{U_\alpha}}^n$; in
  particular, given a label $a\in \scr{B}(U_\alpha )$, we have a
  bijection
  \begin{displaymath}
    \opnm{Hom}_{\scr{B}(U_\alpha )}(a,a)\longrightarrow 
    \opnm{Hom}_{\tsf{LF}_{\scr{O}_{U_\alpha}}^n}(F_\alpha (a),F_\alpha (a)),
  \end{displaymath}
  which is in fact an isomorphism of algebras
  \begin{displaymath}
    \Gamma_{aa}\longrightarrow \bigoplus_k
    \opnm{End}_{\tsf{LF}_{\scr{O}_{U_\alpha}}}\left (\Gamma_{\xi_k a}\right ).
  \end{displaymath}
  We can then assume that the transition homomorphism
  $\iota_a:\scr{T}_{U_\alpha }\to \Gamma_{aa}$ is in fact a morphism
  \begin{displaymath}
    \iota_a:\scr{T}_{U_\alpha }\longrightarrow 
    \bigoplus_k\opnm{End}_{\tsf{LF}_{\scr{O}_{U_\alpha}}}
    \left (\Gamma_{\xi_k a}\right );
  \end{displaymath}
  in other words, the map $\iota_a$ consists of $n$ morphisms
  \begin{displaymath}
    \iota_a^k :\scr{T}_{U_\alpha }\longrightarrow 
    \opnm{End}_{\tsf{LF}_{\scr{O}_{U_\alpha}}}\left (\Gamma_{\xi_k a}\right ).
  \end{displaymath}

  In our case, we have that the morphism $\iota_a$ is central; this
  condition can be also expressed by saying that the morphisms
  $\iota_a^k$ are central ($k=1,\dots ,n$). Hence, for each $k=1,\dots
  ,n$, $\left (\Gamma_{\xi_ka},\iota_a^k\right )$ is a Higgs pair for
  $U_\alpha$.
\end{proof}

We shall now describe the BDR 2-vector bundle structure for the stack
$\scr{B}$ (check definition \ref{bdr_2bundle} for details).

We first point out that, being $M$ paracompact, the open cover by
semisimple open subsets $\mathfrak{U}=\{U_\alpha \}$ can be taken to
be indexed by a poset (which we shall not include in our
notation). For each index $i=1,\dots ,n$, let $\xi^\alpha_i \in
\scr{B}(U_\alpha )$ be a label as in proposition
\ref{invertibles}. Let $U_\beta$ be another semisimple subset such
that $U_{\alpha \beta}\neq \emptyset$ and let $\{e_i^\alpha\}$ and
$\{e_i^\beta\}$ be frames of simple idempotent sections over
$U_\alpha$ and $U_\beta$ respectively. We then have a permutation
$u=u_{\alpha \beta}:\{1,\dots ,n\}\to \{1,\dots ,n\}$ such that, over
$U_{\alpha \beta}$,
$$e_i^\alpha =e_{u(i)}^\beta .$$
By proposition \ref{uniqueness_labels}, the previous equation is
equivalent to the existence of invertible sheaves $\scr{L}^{\alpha
  \beta}_i$ such that, over $U_{\alpha \beta}$,
\begin{displaymath}
  \xi_i^\alpha \cong \scr{L}^{\alpha \beta}_{u(i)}\otimes \xi_{u(i)}^\beta .
\end{displaymath}
Write $\xi^\alpha :=(\xi_1^\alpha ,\dots ,\xi_n^\alpha )^t$. Then, we
can write the previous equation in matrix form
\begin{equation}\label{mult_matrix}
\xi^\alpha \cong A^{\alpha \beta}_u \xi^\beta ,
\end{equation}
where $A^{\alpha \beta}_u$ is a matrix obtained from the diagonal
matrix
\begin{displaymath}
  \opnm{diag}\left (\scr{L}^{\alpha \beta}_1,\dots ,\scr{L}^{\alpha \beta}_n\right )
\end{displaymath}
by applying the permutation $u$ to its columns. Let now $\gamma$ be
such that $U_{\alpha \beta \gamma}\neq \emptyset$ and suppose that the
idempotents are permuted according to $v$ over $U_{\beta \gamma}$ and
$w$ over $U_{\alpha \gamma}$.

\begin{lemma}
  We have an isomorphism $A^{\alpha \beta}_uA^{\beta \gamma}_v\cong
  A^{\alpha \gamma}_w$ (i.e. the corresponding matrix entries on each
  side have isomorphic bundles).
\end{lemma} 
\begin{proof}
  Assume that the idempotents are permuted according to
  \begin{itemize}
  \item $u$ over $U_{\alpha \beta}$,
  \item $v$ over $U_{\beta \gamma}$ and
  \item $w$ over $U_{\alpha \gamma}$.
  \end{itemize}
  Then, by uniqueness, we should have $vu=w$. Now pick a vector
  $\xi^\gamma$. Then, the $i$-th coordinate of $A^{\alpha
    \beta}_uA^{\beta \gamma}_v\xi^\gamma$ is given by
  $
    \scr{L}_i^{\alpha \beta }\otimes \scr{L}_{u(i)}^{\beta \gamma }\otimes \xi_{v(u(i))}^\gamma ,
  $
   and the one corresponding to the product $A^{\alpha
    \gamma}_w\xi^\gamma$ is
  $
    \scr{L}^{\alpha \gamma}_i\otimes \xi_{w(i)}^\gamma .
  $
  As both objects are isomorphic to $\xi_i^\alpha$, they are both
  isomorphic, and hence by \ref{uniqueness_labels},
  \begin{displaymath}
    \scr{L}_i^{\alpha \beta }\otimes \scr{L}_{u(i)}^{\beta \gamma } 
    \cong \scr{L}^{\alpha \gamma}_i,
  \end{displaymath}
  as desired.
\end{proof}

If $A=(E_{ij})$ is an $n\times n$ matrix of vector bundles, we denote
by $\opnm{rk}A\in \opnm{M}_n(\nat_0)$ the matrix which $(i,j)$ entry
is $\opnm{rk}E_{ij}$. Then, by definition,
\begin{displaymath}
  \opnm{det}\left (\opnm{rk} A_u^{\alpha \beta}\right )=\pm 1.
\end{displaymath}
Moreover, associativity of the tensor product renders the following diagram
\begin{displaymath}
  \xymatrix{
    A^{\alpha \beta}(A^{\beta \gamma }A^{\gamma \delta }) \ar[rr] \ar[d] & 
    & (A^{\alpha \beta }A^{\beta \gamma })A^{\gamma \delta } \ar[d] \\
    A^{\alpha \beta }A^{\beta \delta } \ar[r] & A^{\alpha \delta } & A^{\alpha \gamma }A^{\gamma \delta }, \ar[l] 
  }
\end{displaymath}
commutative (see definition \ref{bdr_2bundle}). We can then state the following

\begin{theorem}
  Let $M$ be a semisimple $F$-manifold of dimension $n$. Then, any
  maximal Cardy fibration $\scr{B}$ over $M$ has a canonical BDR 2-vector
  bundle of rank $n$ attached to it.
\end{theorem}

%%% Local Variables: 
%%% mode: latex
%%% TeX-master: "master"
%%% End: 

\medskip
\section{Branes and twisted bundles}
\label{sec:maxim-categ-bran}

Let now $\mathscr{A}$ be an algebra over $M$, i.e. a sheaf of (non
necessarily commutative) $\scr{O}_M$-algebras, and assume also that
$\scr{A}$ is locally-free as an $\scr{O}_M$-module. Let
$
  \iota :\mathscr{T}_M\longrightarrow \mathscr{A}
$
be a central morphism; this map provides $\scr{A}$ with a structure of
$\mathscr{T}_M$-algebra.

\begin{lemma}
  If $S$ is the spectral cover of $M$ with projection $\pi
  :S\rightarrow M$, the topological inverse image $\pi^{-1}\scr{T}$ is
  a sheaf of rings (and of $\pi^{-1}\scr{O}_M$-modules) and
  $\pi^{-1}\scr{A}$ is a $\pi^{-1}\scr{T}$-algebra by means of the
  central morphism $\pi^{-1}\iota :\pi^{-1}\mathscr{T}\longrightarrow
  \pi^{-1}\mathscr{A}$ which is given by
  $
    \pi^{-1}\iota (\sigma )_{\varphi }=\iota_{\pi (y)}(\sigma (y)).
  $
\end{lemma}
\begin{proof}
  Recall that, for a sheaf over $\scr{S}$ over $M$, $\pi^{-1}\scr{S}$
  is the sheaf given by $\pi^{-1}\scr{S}(\widetilde{U})=\scr{S}(\pi
  (\widetilde{U}))$. From this definition, the statement of the lemma
  readily follows.
\end{proof}

In the following we shall consider the ringed space $(S,\scr{O}_S)$
and also $M$ with two different ringed structures: one given by
$\scr{O}_M$ and the other by the sheaf of algebras $\scr{T}$. By
proposition \ref{isom_1}, we have distinguished maps $u_1:\scr{O}_M\to
\pi_*\scr{O}_S$ and $u_2:\scr{T}\to \pi_*\scr{O}_S$, which can be
regarded as the inclusion $f\mapsto f1$ and the identity,
respectively. This maps define two morphisms of ringed spaces $(\pi
,u_1):(S,\scr{O}_S)\to (M,\scr{O}_M)$ and $(\pi ,u_2):(S,\scr{O}_S)\to
(M,\scr{T})$. By the adjunction between $\pi_*$ and $\pi^{-1}$ we have
change-of-ring morphisms
\begin{equation}\label{change_of_rings}
  \pi^{-1}\scr{O}_M\longrightarrow \scr{O}_S \quad 
  \text{and} \quad \pi^{-1}\scr{T}\longrightarrow \scr{O}_S,
\end{equation}
and the inverse images
\begin{displaymath}
  \begin{aligned}
    \pi^*\scr{T} &= \scr{O}_S\otimes_{\pi^{-1}\scr{O}_M}\pi^{-1}\scr{T} \\
    \pi^*\scr{A}   &= \scr{O}_S\otimes_{\pi^{-1}\scr{T}}\pi^{-1}\scr{A} \\
\end{aligned}
\end{displaymath}
are $\scr{O}_S$-algebras. By considering the morphism
$
  \xymatrix{
    \pi^*\scr{T} \ar[rr]^{1\otimes \pi^{-1}\iota} 
    && \pi^*\scr{A} },
$
the sheaf $\pi^*\scr{A}$ turns out to be a $\pi^*\scr{T}$-algebra. The
actions that provide these algebra structures will be described
explicitly after introducing some other tools that we need.

\begin{lemma}
  Let $\scr{A}$ be a sheaf of commutative $\scr{R}$-algebras over $S$,
  where $\scr{R}$ is a sheaf of commutative rings. Then $\pi_*\scr{A}$
  is a sheaf of $\pi_*\scr{R}$-algebras.
\end{lemma}
In what follows, we regard $S$ as being a submanifold of $T^*M$;
i.e. points of $S$ are multiplicative linear maps $\varphi :T_xM\to
\comp$, where $x=\pi (\varphi )$. We now define a global section
$\sigma_0\in \Gamma (S;\pi^{-1}\scr{T})$ in the following way: we let
$\sigma_0:S\to \bigsqcup_{\varphi \in S}\scr{T}_{\pi (\varphi )}$ be
given by
$$\sigma_0(\varphi ):=(\varphi ,e^{\varphi}_x),$$
where $x=\pi (\varphi )$ and $e^{\varphi}_x$ is the germ at $x$ of the
unique idempotent local section $e^{\varphi}:U\to TM$ which verifies
$\varphi (e^{\varphi}(x))=1$. Note that $\sigma_0$ induces a section
$1\otimes \sigma_0\in \Gamma (S;\pi^*\scr{T})$ and, moreover,
$\sigma_0$ as well as $1\otimes \sigma_0$ are idempotent. Likewise,
$\sigma_0$ also induces (global) idempotent sections on
$\pi^{-1}\scr{A}$ and $\pi^*\scr{A}$ given by $\pi^{-1}\iota
(\sigma_0)$ and $1\otimes \pi^{-1}\iota (\sigma_0)$, respectively. To
be more explicit, we have
\begin{displaymath}
  \begin{aligned}
    1\otimes \sigma_0 \in \Gamma (S;\pi^*\scr{T}) \quad , 
    \quad 1\otimes \sigma_0 &:S\longrightarrow \bigsqcup_{\varphi \in S}
    \scr{O}_{S,\varphi}\otimes_{\scr{O}_{M,\pi (\varphi )}}\scr{T}_{\pi (\varphi )}, \\
    \pi^{-1}\iota (\sigma_0) \in \Gamma(S;\pi^{-1}\scr{A}) \quad , 
    \quad \pi^{-1}\iota (\sigma_0) &: 
    S\longrightarrow \bigsqcup_{\varphi \in S}\scr{A}_{\pi (\varphi )}, \\
    1\otimes \pi^{-1}\iota (\sigma_0) \in \Gamma (S;\pi^*\scr{A}) \quad , 
    \quad 1\otimes \pi^{-1}\iota (\sigma_0) &: 
    S\longrightarrow \bigsqcup_{\varphi \in S}
    \scr{O}_{S,\varphi}\otimes_{\scr{T}_{\pi (\varphi )}}\scr{A}_{\pi (\varphi )}, \\
\end{aligned}
\end{displaymath}
given by the following expressions:
\begin{displaymath}
  \begin{aligned}
    (1\otimes \sigma_0)_\varphi &= 1\otimes e^{\varphi}_x ,\\
    \pi^{-1}\iota (\sigma_0)_\varphi &= \iota_x(e^{\varphi}_x), \\
    (1\otimes \pi^{-1}\iota (\sigma_0))_\varphi 
    &= 1\otimes \iota_x(e^{\varphi}_x), \\
\end{aligned}
\end{displaymath}
where $x=\pi (\varphi )$.

\begin{proposition}\label{subsheaf}
  Let $\mathscr{A}$ be an algebra over a space $M$ and let $e\in
  \mathscr{A}(M)$ be a global idempotent section. Then the assignment
  \begin{displaymath}
    U\longmapsto e\scr{A}(U)=\{e\sigma \ | \ \sigma\in\scr{A}(U)\}
  \end{displaymath}
  is a sheaf of ideals.
\end{proposition}
\begin{proof}
  Let $\{U_i\}$ be an open cover of an open subset $U\subset M$; for
  each index $i$, let $\sigma_i\in e\scr{A}(U_i)$ such that
  $\sigma_i=\sigma_j$ over $U_{ij}$. Then we have:
\begin{enumerate}
\item for each $i$, there exists a section $\tau_i\in
  \mathscr{A}(U_i)$ such that $\sigma_i=e\tau_i$ and
\item as $\mathscr{A}$ is a sheaf, there exists a unique section
  $\sigma \in \mathscr{A}(U)$ with $\sigma|_{U_i}=\sigma_i$ for each
  $i$.
\end{enumerate}
Consider now the section $e\sigma \in e\scr{A}(U)$. Then, over $U_i$
we have
\begin{displaymath}
  (e\sigma )|_{U_i}=e\sigma_i=e(e\tau_i)=e\tau_i=\sigma_i,
\end{displaymath}
and thus, by uniqueness, $\sigma =e\sigma \in \mathscr{A}(U)$.
\end{proof}

\begin{notation}
  The sheaves $(1\otimes \sigma_0)\pi^*\scr{T}_M$ and $(1\otimes
  \pi^{-1}\iota (\sigma_0))\pi^*\scr{A}$, will be denoted by
  $\scr{T}^*_0$ and $\scr{A}^*_0$ respectively. The notation
  $\epsilon^{\varphi}_x$ will be adopted for the germ
  $\iota_x(e^{\varphi}_x)$.
\end{notation}

By the previous result, the sheaves $\scr{T}^*_0$ and $\scr{A}^*_0$
are $\scr{O}_S$-algebras and their stalks are given by the expressions
\begin{displaymath}
  \begin{aligned}
    \scr{T}^*_{0,\varphi } &= \scr{O}_{S,\varphi }\otimes_{\scr{O}_{M,x}}e^{\varphi}_x\scr{T}_x, \\
    \scr{A}^*_{0,\varphi } &= \scr{O}_{S,\varphi }\otimes_{\scr{T}_x}\epsilon^{\varphi}_x\scr{A}_x, \\
\end{aligned}
\end{displaymath}
where $x=\pi (\varphi )$.

\begin{notation}
From now on, we will supress the coefficient rings in the notation of the tensor product.
\end{notation}

\begin{proposition}\label{isom_2}
  There exists a canonical isomorphism of $\scr{O}_S$-algebras
  $
    \scr{T}^*_0\stackrel{\cong}{\longrightarrow}\scr{O}_S.
  $
\end{proposition}
\begin{proof}
  The correspondence $\scr{O}_S\to \scr{T}^*_0$ given by
  $
    f\longmapsto f\otimes \sigma_0.
  $
  provides the desired isomorphism.
\end{proof}

Combining \ref{isom_1} and \ref{isom_2} we have the following

\begin{cor}
  There exists a canonical isomorphism of $\scr{O}_M$-algebras
  \begin{displaymath}
    \pi_*\scr{T}^*_0\stackrel{\cong}{\longrightarrow}\scr{T}.
  \end{displaymath}
\end{cor}

\begin{lemma}\label{iota_decomposition}
  If $U\subset M$ is a semisimple neighborhood with basis $\{e_1,\dots
  ,e_n\}$, there exists an isomorphism
  \begin{displaymath}
    \mathscr{A}|_U\cong \bigoplus_i\iota (e_i)\mathscr{A}|_U.
  \end{displaymath}
\end{lemma}
\begin{proof}
  Define $\phi :\scr{A}|_U\to \bigoplus_i\iota (e_i)\scr{A}|_U$ by
  \begin{displaymath}
    \phi (\sigma )=\sum_i\iota (e_i)\sigma .
  \end{displaymath}
  Recalling that the stalk $\Bigl (\bigoplus_i\iota
  (e_i)\scr{A}|_U\Bigr )_x$ is given by
  $\bigoplus_{\varphi}\epsilon^{\varphi}_x\scr{A}_x$, the statement of
  the lemma follows.
\end{proof}

\begin{theorem}\label{isom_3}
  The assignment $\sigma \mapsto \overline{\sigma}$ defines an
  isomorphism of $\scr{T}$-algebras
  \begin{displaymath}
    \scr{A}\longrightarrow \pi_*\scr{A}^*_0.
  \end{displaymath}
\end{theorem}
\begin{proof}
  The equalities $\overline{1}=1$ and $\overline{\sigma
    +\tau}=\overline{\sigma}+\overline{\tau}$ are straightforward to
  verify. Let us now check that $\overline{\sigma
    \tau}=\overline{\sigma}\; \overline{\tau}$ holds. We have
  \begin{displaymath}
    \begin{aligned}
      (\overline{\sigma \tau})_x &= \sum_{\varphi \in \pi^{-1}(x)}1
      \otimes \epsilon^{\varphi}_x\sigma_x\tau_x \\
      &= \sum_{\varphi \in \pi^{-1}(x)}1\otimes 
      \epsilon^{\varphi}_x\sigma_x \epsilon^{\varphi}_x \tau_x \\
      &= \left (\sum_{\varphi \in \pi^{-1}(x)}1\otimes 
        \epsilon^{\varphi}_x\sigma_x\right )
      \left (\sum_{\varphi \in \pi^{-1}(x)}1\otimes 
        \epsilon^{\varphi}_x\tau_x\right )=\overline{\sigma}_x\overline{\tau}_x. \\
    \end{aligned}
  \end{displaymath}
  Let $X$ be a vector field on $M$ with local representation
  $X=\sum_{\varphi \in \pi^{-1}(x)}\lambda_\varphi e^{\varphi}$. We
  will now check that $\overline{X\cdot \sigma}=X\cdot
  \overline{\sigma}$, which is almost a tautology. The left hand side
  is
  \begin{displaymath}
    \begin{aligned}
      (\overline{X\cdot \sigma})_x &= 
      \sum_{\varphi \in \pi^{-1}(x)}1\otimes 
      \lambda_{\varphi ,x}\epsilon^{\varphi}_x\sigma_x.\\
      &= \sum_{\varphi \in \pi^{-1}(x)}
      \widetilde{\lambda}_\varphi \otimes \epsilon^{\varphi}_x\sigma_x,\\
    \end{aligned}
  \end{displaymath}
  where $\widetilde{\lambda}$ is the map on $\pi^{-1}(U)$ defined by
  $\widetilde{\lambda}(\varphi )=\lambda (\pi (\varphi ))$. But the
  right hand side is precisely $(X\cdot \overline{\sigma})_x$.

  We will now prove that the assignment $\sigma \mapsto
  \overline{\sigma}$ is a sheaf isomorphism, so we will check that at
  the level of stalks, the maps $\scr{A}_x\to \left
    (\pi_*\scr{A}^*_0\right )_x$ are bijections.

  Let $\tau_x\in \left (\pi_*\scr{A}^*_0\right )_x$ be given by
  $\tau_x=\sum_{\varphi \in \pi^{-1}(x)}f_{\varphi}\otimes
  \epsilon^{\varphi}_x\sigma_{\varphi ,x}$. Assume also that
  $f_\varphi$ is the germ of a function, which, abusing, we denote
  again by $f_\varphi$, defined in a neighborhood
  $\widetilde{U}_\varphi$ of $\varphi$ such that $\pi
  |_{\widetilde{U}_\varphi}$ is a homeomorphism. If we define
  \begin{displaymath}
    \sigma_x=\sum_{\varphi \in \pi^{-1}(x)}(f_\varphi\pi^{-1})_x
    \epsilon_{\varphi ,x}\sigma_{\varphi ,x}\in \scr{A}_x,
  \end{displaymath}
  then $\sigma_x\mapsto \tau_x$.

  Suppose now that $\overline{\sigma}_x=\sum_{\varphi \in
    \pi^{-1}(x)}1\otimes \epsilon^{\varphi}_x\sigma_x=0$. As all the
  modules (stalks) involved are free, this equality implies
  immediately that $\epsilon^{\varphi}_x\sigma_x=0$ for each $\varphi
  \in \pi^{-1}(x)$, and thus $\sigma_x=0$. This finishes the proof.
\end{proof}

% Hasta linea 200 de anibal 2014-06-20

Recall now that a functor $F:{\bf X}\to {\bf Y}$ is said to be
\emph{essentially surjective} if for each object $Y\in
{\bf Y}$ there exists an object $X\in {\bf X}$ such that $F(X)$ is
isomorphic to $Y$. For a sheaf of rings or algebras $\scr{R}$, we let
$\tsf{Mod}_{\scr{R}}$ denote the category of $\scr{R}$-modules. The
previous results can then be summarized in the following

\begin{theorem}\label{esurjective}
The functor $\pi_*:\tsf{Mod}_{\scr{O}_S}\to \tsf{Mod}_{\scr{T}}$ is
essentially surjective. 
\end{theorem}

\subsection{A Correspondence Between Branes and Twisted Vector
  Bundles}
\label{sec:corr-betw-bran}

Consider now a global label $a\in \scr{B}(M)$; we can then apply the
machinery of the previous sections to the $\scr{T}$-algebra
$\Gamma_{aa}$. Hence, by \ref{esurjective}, there exists an
$\scr{O}_S$-algebra $\widetilde{\Gamma}_{aa}$ such that $\pi_*
\widetilde{\Gamma}_{aa}\cong \Gamma_{aa}$.

\begin{theorem}\label{azumaya_s}
$\widetilde{\Gamma}_{aa}$ is an Azumaya algebra over $S$.
\end{theorem}
\begin{proof}
  Let $x\in M$ and let $U$ be a semisimple neighborhood of $x$, with
  $\pi^{-1}(U)=\bigsqcup_i\widetilde{U}_i$ If $a\in \scr{B}(M)$ is a
  global label, then we can apply \ref{theorem2} to the restriction
  $a|_U$. Let $\{e_1,\dots ,e_n\}$ be a frame of simple, orthogonal
  idempotent sections over $U$. Suppose now that $e_i$ is the section
  corresponding to the sheet $\widetilde{U}_i$. By constructions in
  the previous section, and also theorem \ref{theorem2} and remark
  \ref{remark_summands}, we can write
  \begin{displaymath}
       \widetilde{\Gamma}_{aa}|_{\widetilde{U}_i} 
       = \iota_a(e_i)\Gamma_{aa}|_{\pi (\widetilde{U}_i)} 
      \cong \iota_a(e_i)\Gamma_{aa}|_{U} 
      \cong \opnm{M}_{d(a,i)}(\scr{O}_{U}). 
  \end{displaymath}
\end{proof}

Note that the dimension of the matrix algebras may vary at different
sheets: if $\Gamma_{aa}$ is isomorphic over a semisimple $U$ to
$\bigoplus_{i}\operatorname{M}_{d_i}(\scr{O}_M)$, then, if $\varphi
\in \widetilde{U}$, $\pi (\varphi )=x\in U$ and $\widetilde{U}$ is a
sufficiently small neighborhood around $\varphi$, we have that
\begin{displaymath}
  \widetilde{\Gamma}_{aa}|_{\widetilde{U}}\cong \operatorname{M}_{d_i}(\scr{O}_{\widetilde{U}}).
\end{displaymath}
If the cover $S$ is connected, then this dimension is constant. In
this case, we then have a twisted vector bundle $\mathbb{E}_a$ over
$S$ such that
\begin{displaymath}
  \operatorname{END}(\mathbb{E}_a)\cong \widetilde{\Gamma}_{aa}.
\end{displaymath}
From now on we shall assume that $S$ is connected.

Take now two boundary conditions $a,b\in \scr{B}(M)$ such that
$\Gamma_{aa}\cong \Gamma_{bb}$. On a semisimple open subset $U_i$ we
can represent both labels in the form
\begin{displaymath}
  \begin{aligned}
    a|_{U_i}&=\bigoplus_k\scr{M}_k\otimes \xi_k, \\
    b|_{U_i}&=\bigoplus_k\scr{N}_k\otimes \xi_k, \\
\end{aligned}
\end{displaymath}
where $\scr{M}_k,\scr{N}_k$ are locally free modules and $\xi_k$ are
the objects of proposition \ref{invertibles}. Then,
$\Gamma_{aa}|_{U_i}\cong
\bigoplus_k\underline{\opnm{End}}_{\scr{O}_{U_i}}(\scr{M}_k)$ and
$\Gamma_{bb}|_{U_i}\cong
\bigoplus_k\underline{\opnm{End}}_{\scr{O}_{U_i}}(\scr{N}_k)$. By
theorem \ref{azumaya_s} and the connectivity of $S$ we can write
\begin{equation}\label{global_labels_local_rep}
  \begin{aligned}
    \Gamma_{aa}|_{U_i}&\cong \underline{\opnm{End}}^{\oplus n}_{\scr{O}_{U_i}}(\scr{M}^{(i)}), \\
    \Gamma_{bb}|_{U_i}&\cong \underline{\opnm{End}}^{\oplus n}_{\scr{O}_{U_i}}(\scr{N}^{(i)}). \\
  \end{aligned}
\end{equation}
for some locally free modules $\scr{M}^{(i)}$ and $\scr{N}^{(i)}$ over
$U_i$. As $\Gamma_{aa}$ and $\Gamma_{bb}$ are isomorphic, we can
assure the existence of invertible sheaves $\scr{L}_i$ such that
$\scr{N}^{(i)}\cong \scr{L}_i\otimes \scr{M}^{(i)}$. By shrinking the
open subset if necessary, we can regard these invertible sheaves as
free.

From equations \eqref{global_labels_local_rep} let us denote by
$\widehat{\scr{M}}$ and $\widehat{\scr{N}}$ the locally free sheaves
with local representation
$\underline{\opnm{End}}_{\scr{O}_{U_i}}(\scr{M}^{(i)})$ and
$\underline{\opnm{End}}_{\scr{O}_{U_i}}(\scr{N}^{(i)})$
respectively. Then
\begin{itemize}
\item $\widehat{\scr{M}}$ and $\widehat{\scr{N}}$ are Azumaya
  algebras. Hence, there exist twisted bundles $\mathbb{E}$ and
  $\mathbb{F}$ such that $\widehat{\scr{M}}\cong
  \Gamma_{\operatorname{END}(\mathbb{E})}$ and $\widehat{\scr{N}}\cong
  \Gamma_{\operatorname{END}(\mathbb{F})}$.
\item As $\Gamma_{aa}$ and $\Gamma_{bb}$ are isomorphic,
  $\widehat{\scr{M}}$ and $\widehat{\scr{N}}$ are also isomorphic. In
  particular, $\operatorname{END}(\mathbb{E})$ and
  $\operatorname{END}(\mathbb{F})$ are isomorphic.
\end{itemize}

\begin{proposition}\label{tensor_l}
  Let $\mathbb{E}$ and $\mathbb{F}$ be two twisted bundles over a
  space $M$. Then the algebra bundles $\operatorname{END}(\mathbb{E})$
  and $\operatorname{END}(\mathbb{F})$ are isomorphic if and only if
  there exists a twisted line bundle $\mathbb{L}$ such that
  $\mathbb{F}\cong \mathbb{E}\otimes \mathbb{L}$.
\end{proposition}
\begin{proof}
  We make use of \ref{isomorphic}. Let $\mathbb{E}, \mathbb{F}$ be
  given by
  \begin{displaymath}
    \begin{aligned}
      \mathbb{E} &= (\mathfrak{U},U_i\times \comp^n,g_{ij},\lambda_{ijk}), \\
      \mathbb{F} &= (\mathfrak{U},U_i\times \comp^n,f_{ij},\mu_{ijk}). \\
\end{aligned}
  \end{displaymath}
  For the ``if'' part, let $\mathbb{L}$ be given by
  $(\mathfrak{U},U_i\times \comp,\xi_{ij},\eta_{ijk})$, where
  $\xi_{ij}:U_{ij}\to \comp^\times$. Assume that $u_{ij}:U_{ij}\to
  \operatorname{GL}(\operatorname{M}_n(\comp ))$ are the cocycles for
  $\operatorname{END}(\mathbb{E}\otimes \mathbb{L})$; then,
  \begin{displaymath}
    \begin{aligned}
      u_{ij}(x)(A) &= \xi_{ij}(x)g_{ij}(x)Ag_{ij}(x)^{-1}\xi_{ij}(x)^{-1} \\
             &= g_{ij}(x)Ag_{ij}(x)^{-1}, \\
    \end{aligned}
  \end{displaymath}
  which are precisely the cocycles for $\operatorname{END}(\mathbb{E})$.

  For the ``only if'' part, assume that
  $\operatorname{END}(\mathbb{E})\cong \operatorname{END}(\mathbb{F})$
  and let $\{\alpha_i:U_i\to
  \operatorname{GL}(\operatorname{M}_n(\comp ))\}$ be a family of maps
  as in \ref{isomorphic}. Then, for each $n\times n$ matrix $A$ we
  have
  \begin{displaymath}
    f_{ij}(x)Af_{ij}(x)^{-1}=(\alpha_i(x)g_{ij}(x)\alpha_j(x)^{-1})A(\alpha_i(x)g_{ij}(x)\alpha_j(x)^{-1})^{-1}
  \end{displaymath}
over $U_{ij}$. This equality implies that there exists a map $\xi_{ij}:U_{ij}\to \comp^\times$ such that
\begin{equation}\label{e_times_l}
  f_{ij}(x)^{-1}\alpha_i(x)g_{ij}(x)\alpha_j(x)^{-1}=\xi_{ij}(x)1
\end{equation}
or, equivalently,
\begin{displaymath}
  f_{ij}(x)=\alpha_i(x)\xi_{ij}(x)^{-1}g_{ij}(x)\alpha_j(x)^{-1},
\end{displaymath}
where $\alpha_i(x)$ is regarded here as an invertible matrix (by the
Skolem-Noether theorem).

We now only need to show that $\{\xi_{ij}\}$ is a (twisted)
cocycle. Multiplying equation \eqref{e_times_l} by the one
corresponding to $\xi_{jk}$ and using the twistings for $\mathbb{E}$
and $\mathbb{F}$ (we omit any reference to $x\in U_{ijk}$ for
simplicity) we obtain
\begin{displaymath}
  \alpha_i \lambda_{ijk}g_{ik}\alpha_k^{-1}=\xi_{ij}\xi_{jk}\mu_{ijk}f_{ik};
\end{displaymath}
rearranging the last equation we must have
\begin{displaymath}
  \xi_{ij}\xi_{jk}=\lambda_{ijk}\mu_{ijk}^{-1}\xi_{ik},
\end{displaymath}
as desired.
\end{proof}

Let now $\operatorname{B}(M)/\sim$ be the set of labels over $M$ subject to the identification
\begin{displaymath}
  a\sim b \Longleftrightarrow \Gamma_{aa}\cong \Gamma_{bb}
\end{displaymath}

and let $\operatorname{TVB}(S)$ be the set of twisted vector bundles
over $S$. We can then define a map
\begin{displaymath}
  \Phi :\operatorname{B}(M)/\sim 
  \longrightarrow \operatorname{TVB}(S)/_{\mathbb{E}\sim \mathbb{L}\otimes \mathbb{E}}
\end{displaymath}
by $\Phi (a)=\mathbb{E}_a$, where $\mathbb{L}$ is a twisted line
bundle. The results obtained in the previous paragraphs let us
conclude with the following characterization of branes in terms of
twisted bundles.

\begin{theorem}
The map $\Phi$ is injective.
\end{theorem}

In other words, we can regard each label (up to equivalence) over $M$
as a twisted bundle (again, up to equivalence) over the spectral
cover.

Now, by theorem \ref{bij_tensor}, we have a bijection
\begin{displaymath}
  \Psi :\opnm{TVB}(S)/_{\mathbb{E}\sim \mathbb{L}\otimes \mathbb{E}}\stackrel{\cong}{\longrightarrow}\opnm{Vect}(S)/_{E\sim L\otimes E},
\end{displaymath}
and then every brane $a\in \opnm{B}(M)$ can in fact be taken as a
vector bundle over $S$, up to tensoring with a line bundle.

\medskip

%\section{Cardy fibrations  and Twisted Bundles}
% \label{sec:tvb}

%\input{sec3_v4}

\end{document}